\documentclass[final,reqno]{amsart}

\usepackage{graphicx}
\usepackage{amsmath,amsthm,amsfonts,amscd,amssymb}
\usepackage[color,notref,notcite]{showkeys}
\usepackage{url}
\usepackage{subeqnarray}
\usepackage[color,notref,notcite]{showkeys} 
\definecolor{labelkey}{rgb}{0.6,0,1}
\usepackage{enumitem}
\usepackage{algorithm}
\usepackage{algpseudocode}
\usepackage{citesort}

\theoremstyle{plain}
\newtheorem{theorem}{Theorem}[section]

\newtheorem{lemma}[theorem]{Lemma}

\newtheorem{assumptions}[theorem]{Assumptions}

\newtheorem{corollary}[theorem]{Corollary}

\theoremstyle{definition}
\newtheorem{definition}[theorem]{Definition}

\def\bhyp#1{\begin{equation}\label{#1}\begin{array}{c}}
\def\ehyp{\end{array}\end{equation}}

\newcounter{cst}

\theoremstyle{remark}
\newtheorem{remark}[theorem]{Remark}

\numberwithin{equation}{section}
\numberwithin{figure}{section}

\newcommand{\RR}{{\mathbb R}}
\newcommand{\NN}{{\mathbb N}}

\def\O{\Omega}
\def\dsp{\displaystyle}
\def\bfn{\mathbf{n}}

\def\disc{{\mathcal D}}
\def\mesh{{\mathcal M}}
\newcommand{\polyd}{{\mathcal T}}
\def\edges{{\mathcal E}}
\def\edge{\sigma}
\def\xcv{x_K}
\def\cv{K}

\newcommand{\edgescv}{{{\edges}_K}}  
\newcommand{\edgesext}{{{\edges}_{\rm ext}}} 
\newcommand{\edgesint}{{{\edges}_{\rm int}}} 
\newcommand{\centers}{\mathcal{P}}

\def\dr{\partial}
\newcommand{\centeredge}{\overline{x}_\edge} 

\newcommand{\x}{\pmb{x}}

\newif\ifcorr\corrtrue
\usepackage[normalem]{ulem}
\definecolor{violet}{rgb}{0.580,0.,0.827}

\def\bpsi{{\boldsymbol \psi}}

\newcommand{\ud}{\, \mathrm{d}} 
\def\div{\mathop{\rm div}}

\title[Analysis of schemes for reaction diffusion models]{A unified analysis for reaction diffusion models with application to the spiral waves dynamics of the Barkley model}

\author{Yahya Alnashri}
\address[Yahya Alnashri]{Department of Mathematics, Al-Qunfudah University College, Umm Al-Qura University, Saudi Arabia}
\email{yanashri@uqu.edu.sa}

\author{Hasan Alzubaidi}
\address[Hasan Alzubaidi]{Department of Mathematics, Al-Qunfudah University College, Umm Al-Qura University, Saudi Arabia}
\email{yanashri@uqu.edu.sa}
\email{hmzubaidi@uqu.edu.sa}



\date{\today}

\begin{document}
\newcommand{\subscript}[2]{$#1 _ #2$}

\begin{abstract}
Applying the gradient discretisation method (GDM), the paper develops a comprehensive numerical analysis for the reaction diffusion model. Using only three properties, this analysis provides convergence results for several conforming and non-conforming numerical schemes that align with the GDM. As an application of this analysis, the hybrid mimetic mixed (HMM) method for the reaction diffusion model is designed and its convergence established. Numerical experiments using the HMM method are presented to facilitate study of the creation of spiral waves in the Barkley model and the ways in which the waves behave when interacting with the boundaries of their generating medium.
\end{abstract}

\maketitle


\section{Introduction}
\label{introduction}

\par Some arrhythmias can have fatal consequences if they are permitted to remain in cardiac tissues without undergoing treatment \cite{1,2}. One method by which it may be possible to develop a better understanding of how arrhythmias develop and the most effective means by which they can be treated mathematical modelling. A number of scholars have examined the mathematical modelling approaches that are available and, following the application of these methods, they have concluded that cardiac arrhythmias can be traced back to the free rotation of spiral waves. These waves, reactivate an area of tissue at a higher frequency than that associated with the normal sinoatrial node and, thereby, generate a higher than average heartbeat \cite{1,3,4,5,6}. In some cases, the spiral waves break up into smaller spiral waves that stimulate uncoordinated heart contractions that are referred to as fibrillation. When fibrillation develops in the ventricles of the heart, it causes the heart to tremor and its strength diminishes. Lacking the ability to pump blood at a consistent rate, the heart ultimately suffers from a cardiac arrest \cite{7}.

\par It is possible to study the development and behaviour of spiral waves as a form of spatio-temporal solution to reaction-diffusion equations of the general form \cite{8,9},
\begin{align}
\partial_t \bar u(\x,t) -\mu \div(\nabla \bar u(\x,t))&=f(\bar u,\bar v), \quad (\x,t) \in \O\times (0,T),\label{rm1}\\
\partial_t \bar v(\x,t)&=g(\bar u,\bar v) \quad (\x,t) \in \O\times (0,T\label{rm2}),\end{align}
for diffusion coefficient $\mu$,  with initial conditions and a pure Neumann boundary condition are defined on the domain $\O \subset \RR^d, d\geq 1$. 

\par Although the model has been widely studied theoretically since it enters the general theory of reaction diffusion equations \cite{45}, it is typically very rare for researchers to be able to solve models that pertain to waves of this nature. As such, there is an inherent need for numerical solutions that are accurate and reliable. The existing literature describes the dynamics of spiral waves using the reaction-diffusion models in depth. A variety of mathematical approaches have been employed including finite difference methods (FDM) \cite{10,11,12,13}, pseudo-spectral methods \cite{14,15,16}, finite element methods (FEM) \cite{17,18,19}, and finite volume methods (FVM) \cite{20,21,22,23,24,25,26}. In this paper, we will focus on the literature related to FVM. In this regard, the 1997 work of Harrild and Henriquez \cite{20} is a particular note. Harrild and Henriquez employed the elements of a FVM-based formulation to examine the phenomenon of conduction in cardiac tissue. Trew et al. \cite{21} described the creation of a novel finite volume method that could effectively be employed to model bidomain electrical activation in discontinuous cardiac tissue. Coudire and Pierre\cite{22} examined a 3D FVM that could be effectively used to calculate the electrical activity that could be observed in the myocardium on unstructured meshes and identified stability conditions for two time-stepping methods in distinctive settings. They also generated error estimates by which effective solutions of the monodomain reaction-diffusion systems could be estimated. Bendahmane and Karlsen \cite{23} successfully merged a finite volume scheme with Dirichlet boundary conditions for the bidomain model, representing a degenerate reaction-diffusion system that models the electrophysiological waves that can be observed in cardiac tissue. Bendahmane et al. \cite{24} employed a comparable FVM with Neumann boundary conditions to demonstrate the existence and uniqueness of the approximate solution for the monodomain and bidomain models for the myocardial tissue electrical activity. Burger et al. \cite{25} presented some of fully space-time adaptive multiresolution approaches that were based on a combination of FVM and Barkley's approach for modeling the complex dynamics of waves in excitable media.  More recently in 2017, Coudire and Turpault \cite{26}, originated and evaluated a high-order FVM approach in space combined with a high-order strong stability preserving (SSP) Runge--Kutta technique to produce 2D simulations of spiral waves and simple planar waves in cardiac tissue. Their research also yielded error estimates for the given reaction-diffusion systems.

\par The chief elements of the spiral waves dynamics such as meandering and drift under external perturbations are of great interest. External perturbations
can include light-induced drift, dual spiral interaction, and interaction of spirals with a boundary \cite{27}. When we scrutinize spiral-boundary interactions, we can see two varieties of interaction, with the spiral waves being either reflected or annihilated \cite{27}. Examining spiral-boundary interactions and what occurs in order to annihilate the spiral can assist us in creating techniques to prevent arrhythmia generated by the spiral behaviour. For more details on spiral drift due to the boundary effects, see \cite{27,28,29} and the references given there.

\par This paper's main purpose is to use the gradient discretisation method (GDM) to offer a complete and unified convergence analysis of the reaction-diffusion model. To our knowledge, this analysis provides the first results that are applicable to several conforming and non conforming methods. The GDM is an abstract setting to study the numerical analysis for linear or non linear, steady or time--dependent diffusion partial differential equations (PDEs). Based on a limit number of properties, the GDM can establish convergence of numerical schemes for various models with different boundary conditions. Various studies have established that the GDM covers several families of numerical schemes: conforming, non conforming and mixed finite elements methods (including the non-conforming, Crouzeix Raviart method and the Raviart Thomas method), the discontinuous Galerkin scheme, the vertex approximate gradient (VAG) scheme, hybrid mimetic mixed methods (which contain hybrid mimetic finite differences, hybrid finite volumes/SUSHI scheme and mixed finite volumes), nodal mimetic finite differences, and finite volumes methods (such as some multi-points flux approximation and discrete duality finite volume methods). For more details, see \cite{31,32,33,34,35,36,37,38} and the monograph \cite{30} for a complete presentation.

\par This paper is organised as follows. Section \ref{sec-continous} is devoted to the continuous model. Section \ref{sec-disc-pblm} provides the discrete elements of the GDM and two properties required to analyse the studied model. It also states the discrete problem (the gradient scheme) followed by the main results: existence and uniqueness of the solution to the scheme, and its convergence to the weak solution of the studied model. Section \ref{sec:HMM} explains the role that the GDM plays in applying the the hybrid mixed mimetic (HMM) methods, and establishing its convergence for the considered model. Section \ref{sec-proof} is devoted to proving the main novelty of this paper, Theorem \ref{theorem-conver-rm}, which is obtained via the compactness technique. In Section \ref{sec-numerical} we study numerically using the HMM, the propagation of the spiral waves for the Barkley model and their behavior when they interact with the boundary of the medium where they spread.
\section{Continuous model}\label{sec-continous}
We consider the following system of partial differential equations, representing varieties models such as the Barkley model \cite{10}, the FitzHugh--Nagumo model \cite{39}, the Aliev and Panfilov model \cite{40} and the Belousov--Zhabotinsky reaction \cite{41}. 
\begin{align}
\partial_t \bar u(\x,t) -\mu \div(\nabla \bar u(\x,t))&=f(\bar u,\bar v), \quad (\x,t) \in \O\times (0,T),\label{rm-strong1}\\
\partial_t \bar v(\x,t)&=g(\bar u,\bar v), \quad (\x,t) \in \O\times (0,T),\label{rm-strong2}\\
\nabla\bar u(\x,t)\cdot \mathbf{n}&= 0, \quad (\x,t) \in\partial\O\times (0,T), \label{rm-strong3}\\
\bar u(\x,0)&=u_{\rm ini}(\x), \quad \x \in \O, \label{rm-strong4}\\
\bar v(\x,0)&=v_{\rm ini}(\x), \quad \x \in \O, \label{rm-strong5}
\end{align}
where $\mathbf{n}$ is the outer normal to $\partial\O$. In this excitable media system, the unknowns $\bar u(\x,t)$ and $\bar v(\x,t)$ denote the excitation and recovery terms, respectively. A dynamic of particular case of this model is explained in Section \eqref{sec-numerical}.

Our analysis focuses on the weak formulation of the above reaction diffusion model. Let us assume the following properties on the data of the model.
\begin{assumptions}\label{assump-rm}
The assumptions on the data in Problem \eqref{rm-strong1}--\eqref{rm-strong5} are the following:
\begin{enumerate}
\item $\O$ is an open bounded connected subset of $\RR^d\; (d \geq 1)$, with a Lipschitz boundary, $T >0$, $\mu \in \RR$,
\item $(u_{\rm ini},v_{\rm ini})$ are in $H^1(\O) \times L^\infty(\O)$,
\item the functions $f,\; g: \RR^2 \to \RR$ are in $L^2(\O\times(0,T))$ and  both polynomial functions satisfying the following conditions:
\begin{itemize}
\item they are affine with respect to $\xi$;
\begin{equation}
f(s,\xi)=f_1(s) + f_2(s)\xi, \quad g(s,\xi)=g_1(s) + \alpha \xi,
\end{equation}
where $f_1, f_2, g_1$ are continuous functions defined on $\RR$ and $\alpha$ is a constant.
\item there exists constants $c_i \geq 0 (i = 1...5)$ such that for any $s\in \RR$,
\begin{subequations}
\begin{equation}
|f_1(s)| \leq c_1 + c_2|s|,
\end{equation}
\begin{equation}
|f_2(s)| \leq c_3,
\end{equation}
\begin{equation}
|g_1(s)| \leq c_4 + c_5|s|.
\end{equation}
\end{subequations}

\end{itemize}

\end{enumerate}
\end{assumptions}

\begin{definition}[Weak solution]\label{def-weak}
Under Assumptions \ref{assump-rm}, a pair $(\bar u,\bar v)$ is said to be a weak solution of Problem \eqref{rm-strong1}--\eqref{rm-strong5} if the following properties and equalities hold:
\begin{subequations}\label{rm-weak}
\begin{equation*}
\begin{aligned}
&\bar u \in L^2(0,T;H^1(\O)) \cap C([0,T];L^2(\O)),\; \bar v\in L^2(0,T;L^2(\O)),\\
&\partial_t \bar u \in L^2(\O\times(0,T)),\; \bar u(\cdot,0)=u_{\rm ini},\\
&\mbox{and for all $\bar\varphi \in L^2(0,T;H^1(\O))$ and for all $\bar\psi \in L^2(0,T;L^2(\O))$},\\
&\mbox{such that $\partial_t \bar\psi \in L^2(0,T;L^2(\O))$ and $\bar\psi(\cdot,T)=0$},
\end{aligned}
\end{equation*}
\begin{equation}\label{rm-weak1}
\begin{aligned}
&\dsp\int_0^T \dsp\int_\O  \partial_t \bar u(\x,t) \bar\varphi(\x,t) \ud \x  \ud t 
+\mu\dsp\int_0^T\int_\O \nabla \bar{u}(\x,t) \cdot \nabla \bar\varphi(\x,t)\ud \x \ud t,\\
{}&\quad\quad\quad\quad\quad= \dsp\int_0^T\dsp\int_\O f(\bar u,\bar v)\bar\varphi(\x,t) \ud \x \ud t,
\end{aligned}
\end{equation}
\begin{equation}\label{rm-weak2}
\begin{aligned}
&-\dsp\int_0^T \dsp\int_\O  \bar v(\x,t)  \partial_t\bar\psi(\x,t) \ud \x \ud t 
-\dsp\int_\O v_{\rm ini}(\x)\bar\psi(\x,0) \ud x\\
&\qquad\quad\quad\quad\quad= \dsp\int_0^T\dsp\int_\O g(\bar u,\bar v)\bar\psi(\x,t) \ud \x \ud t.
\end{aligned}
\end{equation}
\end{subequations}
\end{definition}

\begin{remark}
The existence of at least one solution $(\bar u,\bar v)$ to \eqref{rm-weak} will be a consequence of the convergence analysis of the gradient discretisation method.
\end{remark}

\section{Discrete Problem}\label{sec-disc-pblm}
As stated in the introduction section, the analysis of numerical schemes for the approximation of solutions to the reaction diffusion model is performed using the gradient discretisation method. This method first requires the reconstruction of a set of discrete spaces and operators, which is called gradient discretisation (GD).

\begin{definition}[gradient discretisation for reaction diffusion model]\label{def-gd-rm}
Let $\O$ be an open subset of $\RR^d$ (with $d \geq 1$) and $T>0$. A gradient discretisation for the reaction--diffusion model is $\disc=(X_{\disc}, Y_\disc, \Pi_\disc, \Pi_{\disc^{'}}, \nabla_\disc, J_\disc, J_{\disc^{'}},  (t^{(n)})_{n=0,...,N}) )$, where
\begin{itemize}
\item the two set of discrete unknowns $X_{\disc}$ and $Y_\disc$ are finite dimensional vector spaces on $\RR$,
\item the function reconstruction $\Pi_\disc : X_{\disc} \to L^2(\O)$ is a linear,
\item the function reconstruction $\Pi_{\disc^{'}} : Y_{\disc} \to L^2(\O)$ is a linear, and must be defined so that $|| \Pi_{\disc^{'}}\cdot ||_{L^2(\O)}$ is a norm on $Y_\disc$,
\item the gradient reconstruction $\nabla_\disc : X_{\disc} \to L^2(\O)^d$ is a linear and must be defined so that
\begin{equation}\label{norm-disc}
\| \varphi ||_\disc = || \Pi_\disc \varphi ||_{L^2(\O)}+ || \nabla_\disc \varphi ||_{L^2(\O)^d} 
\end{equation}
is a norm on $X_{\disc}$. 
\item $J_\disc: H^1(\O) \to X_\disc$ and $J_{\disc^{'}}: L^\infty(\O) \to Y_\disc$ are a linear and continuous interpolation operator for the
 initial conditions,
\item $t^{(0)}=0<t^{(1)}<....<t^{(N)}=T$ are discrete times. 
\end{itemize}
\end{definition}


Let us introduce some notations to define the space--time reconstructions $\Pi_\disc \varphi: \O\times[0,T]\to \RR$, and $\nabla_\disc \varphi: \O\times[0,T]\to \RR^d$, and the discrete time derivative $\delta_\disc \varphi : (0,T) \to L^2(\O)$, for $\varphi=(\varphi^{(n)})_{n=0,...,N} \in X_\disc^{N}$ and $\psi=(\psi^{(n)})_{n=0,...,N} \in Y_\disc^{N}$ 

For a.e $\x\in\O$, for all $n\in\{0,...,N-1 \}$ and for all $t\in (t^{(n)},t^{(n+1)}]$, let
\begin{equation*}
\begin{split}
&\Pi_\disc \varphi(\x,0)=\Pi_\disc \varphi^{(0)}(\x), \quad \Pi_\disc \varphi(\x,t)=\Pi_\disc \varphi^{(n+1)}(\x),\\
&\nabla_\disc \varphi(\x,t)=\nabla_\disc \varphi^{(n+1)}(\x),\\
&\Pi_{\disc^{'}} \psi(\x,0)=\Pi_{\disc^{'}} \psi^{(0)}(\x), \mbox{ and } \Pi_{\disc^{'}} \psi(\x,t)=\Pi_{\disc^{'}} \psi^{(n+1)}(\x).
\end{split}
\end{equation*}
Set $\delta t^{(n+\frac{1}{2})}=t^{(n+1)}-t^{(n)}$ and $\delta t_\disc=\max_{n=0,...,N-1}\delta t^{(n+\frac{1}{2})}$, to define
\begin{equation*}
\delta_\disc \varphi(t)=\delta_\disc^{(n+\frac{1}{2})}\varphi:=\frac{\Pi_\disc(\varphi^{(n+1)}-\varphi^{(n)})}{\delta t^{(n+\frac{1}{2})}} \mbox{ and }
\delta_{\disc^{'}} \psi(t)=\delta_{\disc^{'}}^{(n+\frac{1}{2})}\psi:=\frac{\Pi_{\disc^{'}}(\psi^{(n+1)}-\psi^{(n)})}{\delta t^{(n+\frac{1}{2})}}.
\end{equation*}

In order to build converging schemes for models including homogenous Neumann boundary conditions, the elements of a GD must enjoy as much as possible the properties of the continuous space and operators, coercivity, consistency and limit-conformity. In our model, there is no need to the coercivity property, which gives the discrete Poincare inequality, since the time derivative in our model plays the same role in establishing $L^2$ estimate on the discrete solution.

A consistency of the method (refers to an interplant error), which ensures the accuracy of approximating smooth functions and their gradients by elements defined from the discrete space, and the convergence of the time steps to zero. 

\begin{definition}[Consistency]\label{def:cons-rm}
If $\disc$ is a gradient discretisation in the sense of Definition \ref{def-gd-rm}, define
$S_{\mathcal{D}} : H^1(\O)\to [0, +\infty)$ and $ S_{\disc^{'}} : L^2(\O)\to [0, +\infty)$ by
\begin{equation}\label{cons-rm1}
\forall \varphi\in H^1(\O), \; S_{\mathcal{D}}(\varphi)= 
\min_{w\in X_\disc}\left(\| \Pi_{\mathcal{D}} w - \varphi \|_{L^{2}(\Omega)}
+ \| \nabla_{\mathcal{D}} w - \nabla \varphi \|_{L^{2}(\Omega)^{d}}\right),
\end{equation}
and
\begin{equation}\label{cons-rm2}
\forall \psi\in L^2(\O), \; S_{\disc^{'}}(\psi)= 
\min_{w\in Y_\disc} \| \Pi_{\disc^{'}} w - \psi \|_{L^{2}(\Omega)}.
\end{equation}
A sequence $(\mathcal{D}_{m})_{m \in \mathbb{N}}$ of gradient discretisations is \emph{consistent} if, as $m \to \infty$ 
\begin{itemize}
\item for all $\varphi \in H^1(\O)$, $S_{\disc_m}(\varphi) \to 0$, and for all $\psi \in L^2(\O)$, $ S_{\disc_m^{'}}(\psi) \to 0$,
\item for all $w \in L^2(\O)$, $\Pi_{\disc_m}J_{\disc_m}w \to w$ in $L^2(\O)$,
\item for all $w\in L^\infty(\O)$, $(\Pi_{\disc_m^{'}}J_{\disc_m^{'}}w)_{m\in\NN}$ is bounded in $L^\infty(\O)$ and converges to $w$ in $L^2(\O)$,
\item $\delta t_{\disc_m} \to 0$.
\end{itemize}
\end{definition}
\par The quantity $W_\disc$ defined below measures how well the discrete Stokes formula is satisfied, it is only exactly ensured in conforming methods.
\begin{definition}[Limit-conformity]\label{def:lconf-rm}
If $\disc$ is a gradient discretisation in the sense of Definition \ref{def-gd-rm}, and $H_{\rm div}=\{\bpsi \in L^2(\O)^d\;:\; {\rm div}\bpsi \in L^2(\O),\; \bpsi\cdot\bfn=0 \mbox{ on } \partial\O  \}$, define $W_{\mathcal{D}} : H_{\rm div} \to [0, +\infty)$ by
\begin{equation}\label{long-rm}
\forall \bpsi \in H_{\rm div},\quad W_{\mathcal{D}}(\bpsi)
 = \sup_{w\in X_\disc\setminus \{0\}}\frac{\Big|\dsp\int_{\Omega}(\nabla_{\mathcal{D}}w\cdot \bpsi + \Pi_{\mathcal{D}}w \div (\bpsi)) \ud x \Big|}{|| w||_\disc }.
\end{equation}
A sequence $(\disc_m)_{m\in \NN}$ of gradient discretisations is \emph{limit-conforming} if for all $\bpsi \in H_{\rm div}$, $W_{\disc_m}(\bpsi) \to 0$, as $m \to \infty$.
\end{definition}
Finally, dealing with non linearity requires the operators $\Pi_\disc$ and $\nabla_\disc$ to provide the compactness properties, defined below. Since $f$ and $g$ are assumed to be affine functions, there is no need for the operator $\Pi_{\disc^{'}}$ to afford this property.    
\begin{definition}[Compactness]\label{def:compact}
A sequence of gradient discretisation $\disc_m$ in the sense of Definition \ref{def-gd-rm} is \emph{compact} if for any sequence $(\varphi_m )_{m\in\NN} \in X_{\disc_m}$, such that $(|| \varphi_m ||_{\disc_m})_{m\in \NN}$ is bounded, the sequence $(\Pi_{\disc_m}\varphi_m )_{m\in \NN}$ is relatively compact in $L^2(\O)$.
\end{definition}

\par Setting the gradient discretisation defined previously in the place of the continuous space and operators in the weak formulation of the model leads to a numerical scheme, called a gradient scheme (GS).

\begin{definition}[Gradient scheme for \eqref{rm-weak}]\label{def-gs-rm} Find families $(u^{(n)})_{n=0,...,N} \in X_\disc^{N+1}$ and $(v^{(n)})_{n=0,...,N} \in Y_\disc^{N+1}$, such that $u^{(0)}=J_\disc u_{\rm ini}$ and $v^{(0)}=J_{\disc^{'}} v_{\rm ini}$, and for all $n=0,...,N-1$, $u^{(n+1)}$ and $v^{(n+1)}$ satisfy
\begin{subequations}\label{rm-disc-pblm}
\begin{equation}\label{rm-disc-pblm1}
\begin{aligned}
&\dsp\int_\O \delta_\disc^{(n+\frac{1}{2})} u(\x) \Pi_\disc \varphi(\x)
+ \mu\dsp\int_\O \nabla_\disc u^{(n+1)}(\x) \cdot \nabla_\disc \varphi(\x)\ud \x\\
&\qquad=\dsp\int_\O f_\disc^{(n+1)}\Pi_\disc \varphi(\x) \ud \x, \quad \forall \varphi \in X_\disc,
\end{aligned}
\end{equation}
\begin{equation}\label{rm-disc-pblm2}
\dsp\int_\O \delta_{\disc^{'}}^{(n+\frac{1}{2})} v(\x) \Pi_{\disc^{'}} \psi(\x) = \dsp\int_\O g_\disc^{(n+1)}\Pi_{\disc^{'}} \psi(\x) \ud \x ,\quad \forall \psi \in Y_\disc,
\end{equation}
\end{subequations}
where the discrete reaction terms are defined by
\[f_\disc^{(n+1)}=f(\Pi_\disc u^{(n+1)}(\x), \Pi_{\disc^{'}} v^{(n+1)}(\x)) \mbox{ and }
g_\disc^{(n+1)}=g(\Pi_\disc u^{(n+1)}(\x), \Pi_{\disc^{'}} v^{(n+1)}(\x)).
\]
\end{definition}

The convergence results of this gradient scheme is stated in the following theorem, whose proof is detailed in Section \ref{sec-proof}.

\begin{theorem}[{\bf Convergence of the GS for the reaction diffusion model}]
\label{theorem-conver-rm} Assume \eqref{assump-rm} and Let $(\disc_m)_{m\in\NN}$ be a sequence of gradient discretisations in the sense of Definition \ref{def-gd-rm}, that is consistent, limit-conforming and compact in the sense of Definitions \ref{def:cons-rm}, \ref{def:lconf-rm} and \ref{def:compact}. For $m \in \NN$, let $(u_m,v_m)$ be a solution to the gradient scheme \eqref{rm-disc-pblm} with $\disc=\disc_m$ and $||\nabla_{\disc_m}J_{\disc_m}u_{\rm ini}||_{L^2(\O)^d}$ is bounded. Then there exists a weak solution $(\bar u,\bar v)$ of \eqref{rm-weak} and a subsequence of gradient discretisations, still denoted by $(\disc_m)_{m\in\NN}$, such that, as $m \to \infty$,
\begin{enumerate}
\item $\Pi_{\disc_m} u_m$ converges strongly to $\bar u$ in $L^\infty(0,T;L^2(\O))$,
\item $\Pi_{\disc_m^{'}} v_m$ converges weakly to $\bar v$ in $L^\infty(0,T;L^2(\O))$,
\item $\nabla_{\disc_m}u_m$ converges strongly to $\nabla\bar u$ in $L^2(\O\times(0,T))^d$. 
\end{enumerate}
\end{theorem}


\section{The hybrid mixed mimetic (HMM) methods}\label{sec:HMM}
We show here that hybrid mixed mimetic (HMM) methods can be expressed as gradient schemes formulation when applied to the reaction diffusion model. It is shown in \cite{42} that the HMM methods is a generic setting gathering three different families of methods, namely the hybrid finite volume method, the (mixed-hybrid) mimetic finite differences methods, and the mixed finite volume methods. In order to construct this mesh--based method, we recall here the notion of polytopal mesh \cite{42}.  

\begin{definition}[Polytopal mesh]\label{def:polymesh}~
Let $\Omega$ be a bounded polytopal open subset of $\RR^d$ ($d\ge 1$). 
A polytopal mesh of $\O$ is given by $\polyd = (\mesh,\edges,\centers)$, where:
\begin{enumerate}
\item $\mesh$ is a finite family of non empty connected polytopal open disjoint subsets of $\O$ (the cells) such that $\overline{\O}= \dsp{\cup_{K \in \mesh} \overline{K}}$.
For any $K\in\mesh$, $|K|>0$ is the measure of $K$ and $h_K$ denotes the diameter of $K$.

\item $\edges$ is a finite family of disjoint subsets of $\overline{\O}$ (the edges of the mesh in 2D,
the faces in 3D), such that any $\edge\in\edges$ is a non empty open subset of a hyperplane of $\RR^d$ and $\edge\subset \overline{\O}$.
We assume that for all $K \in \mesh$ there exists  a subset $\edgescv$ of $\edges$
such that $\dr K  = \dsp{\cup_{\edge \in \edgescv}} \overline{\edge}$. 
We then set $\mesh_\edge = \{K\in\mesh\,:\,\edge\in\edgescv\}$
and assume that, for all $\edge\in\edges$, $\mesh_\edge$ has exactly one element
and $\edge\subset\partial\O$, or $\mesh_\edge$ has two elements and
$\edge\subset\O$. 
$\edgesint$ is the set of all interior faces, i.e. $\edge\in\edges$ such that $\edge\subset \O$, and $\edgesext$ the set of boundary
faces, i.e. $\edge\in\edges$ such that $\edge\subset \dr\O$.
For $\edge\in\edges$, the $(d-1)$-dimensional measure of $\edge$ is $|\edge|$,
the centre of mass of $\edge$ is $\centeredge$, and the diameter of $\edge$ is $h_\edge$.

\item $\centers = (x_K)_{K \in \mesh}$ is a family of points of $\O$ indexed by $\mesh$ and such that, for all  $K\in\mesh$,  $\xcv\in K$ ($\xcv$ is sometimes called the ``centre'' of $\cv$). 
We then assume that all cells $K\in\mesh$ are  strictly $\xcv$-star-shaped, meaning that 
if $x\in \overline{K}$ then the line segment $[\xcv,x)$ is included in $K$.
\end{enumerate}
For a given $K\in \mesh$, let $\bfn_{K,\sigma}$ be the unit vector normal to $\sigma$ outward to $K$
and denote by $d_{K,\sigma}$ the orthogonal distance between $x_K$ and $\sigma\in\mathcal E_K$.
The size of the discretisation is $h_\mesh=\sup\{h_K\,:\; K\in \mesh\}$.
\end{definition}

Let $\polyd$ be a polytopal mesh. A gradient discretisation $(X_\disc, Y_\disc, \Pi_\disc, \Pi_{\disc^{'}}, \nabla_\disc,J_\disc, J_{\disc^{'}})$ in the setting of HMM method formats are then constructed by setting
\[
\begin{aligned}
&X_{\disc}=\{ \varphi=((\varphi_{K})_{K\in \mathcal{M}}, (\varphi_{\sigma})_{\sigma \in \mathcal{E}})\;:\; \varphi_{K} \in \RR,\, \varphi_{\sigma} \in \RR
\},\\
&Y_{\disc}=\{ \varphi=(\varphi_{K})_{K\in \mathcal{M}}\;:\; \varphi_{K} \in \RR
\},\\
&\forall K\in\mathcal M\,:\,\Pi_\disc \varphi=\Pi_{\disc^{'}}\varphi=\varphi_K\mbox{ on $K$},\\
&\forall w \in L^2(\O)\;:\; J_\disc w=((w_K)_{K\in \mesh}, (w_\edge)_{\edge\in \edges_K}) \in X_\disc,\\
&\forall w \in L^\infty(\O),\; J_{\disc^{'}} w=(w_K)_{K\in \mesh} \in Y_\disc,\\
&\mbox{where } w_k = \frac{1}{|K|}\int_K w(\x) \ud \x \mbox{ and } w_\edge =0,\\
&\forall \varphi\in X_\disc,\; \forall K\in\mathcal M,\,\forall \sigma\in\mathcal E_K,\\
&\nabla_\disc \varphi=\nabla_{K}\varphi+
\frac{\sqrt{d}}{d_{K,\sigma}}R_K(\varphi)\mathbf{n}_{K,\sigma} \mbox{ on } D_{K,\edge},
\end{aligned}
\]
where where a cell--wise constant gradient $\nabla_K(\varphi)$ and a stabilisation term $R_K(\varphi)$ are respectively defined by:
\[
\nabla_{K}\varphi= \frac{1}{|K|}\sum_{\sigma\in \edgescv}|\sigma|\varphi_\edge\mathbf{n}_{K,\sigma} \mbox{ and } R_K(\varphi)=(\varphi_\edge - \varphi_K - \nabla_K \varphi\cdot(\centeredge-x_K))_{\edge\in\edges_K}.\\
\]
 
The gradient scheme \eqref{rm-disc-pblm} coming from such a GD can be given by: find $(u^{(n)})_{n=0,...,N} \in X_\disc^{N+1}$ and $(v^{(n)})_{n=0,...,N} \in Y_\disc^{N+1}$, such that $u^{(0)}=J_\disc u_{\rm ini}$ and $v^{(0)}=J_{\disc^{'}} v_{\rm ini}$, and for all $n=0,...,N-1$, $u^{(n+1)}$ and $v^{(n+1)}$ satisfy, for all $\varphi \in X_\disc$ and for all $\psi \in Y_\disc$,
\begin{equation}\label{rm-hmm}
\left.
\begin{aligned}
&\dsp\sum_{K\in \mesh}\frac{|K|}{\delta t^{(n+\frac{1}{2})}}\Big( u_K^{(n+1)}-u_K^{(n)} \Big)\varphi_K
+\dsp\sum_{K\in \mesh}|K|\Lambda_K\nabla_K u^{(n+1)}\cdot \nabla_K \varphi\\
&\quad+\dsp\sum_{K\in \mesh}(R_K\varphi)^T \mathbb B_K R_K (u^{(n+1)})
=\dsp\sum_{K\in \mesh}\varphi_K\dsp\int_K f(u_K^{(n+1)},v_K^{(n+1)}) \ud \x,\\
&\dsp\sum_{K\in \mesh}\frac{|K|}{\delta t^{(n+\frac{1}{2})}}\Big( v_K^{(n+1)}-v_K^{(n)} \Big)\psi_K
=\dsp\sum_{K\in \mesh}\psi_K\dsp\int_K g(u_K^{(n+1)},v_K^{(n+1)}) \ud \x,
\end{aligned}
\right.
\end{equation}
where $\mathbb B_K$ is a symmetric positive definite matrix of size $\mbox{Card}(\edges_K)$.

The  above HMM scheme can be formulated as a finite volume scheme (a format used in the computational processes). To do so, introduce the linear fluxes $u\mapsto F_{K,\sigma}(u)$ (for $K\in\mesh$ and $\sigma\in\edges_K$) defined in \cite{38}:
for all $K \in \mesh$ and all $u,w \in X_\disc$,
\begin{align*}
\sum_{\sigma \in \mathcal{E}_K}|\sigma| F_{K,\sigma}(u)
(w_K-w_\sigma)={}&\mu\int_K \nabla_\disc u\cdot\nabla_\disc w \ud \x.
\end{align*}

Then Problem \eqref{rm-hmm} can be recast as, for all $n=0,...,N-1$,
\begin{align*}
\frac{|K|}{\delta t^{(n+\frac{1}{2})}} \Big( u^{(n+1)}-u^{(n)} \Big)&+\sum_{\sigma \in \mathcal{E}_K}|\sigma|F_{K,\sigma}(u^{(n+1)})= |K| f_K^{(n+1)}, \quad \forall K \in \mesh\\
\frac{|K|}{\delta t^{(n+\frac{1}{2})}} \Big( v^{(n+1)}-v^{(n)} \Big)
&=|K|g_K^{(n+1)}, \quad \forall K \in \mesh\\
F_{K,\sigma}(u^{(n+1)})+F_{L,\sigma}(u^{(n+1)})&= 0, \quad \forall \sigma\in\mathcal E_{\rm int}\mbox{ with }
\mesh_\sigma=\{K,L\},\\
F_{K,\sigma}(u^{(n+1)})&=0, \quad \forall K \in \mesh\,,\forall \sigma \in \mathcal{E}_K \mbox{ such that }\sigma\subset  \partial\O,
\end{align*}
where 
$f_K^{(n+1)}=\frac{1}{|K|}\int_K f(u_K^{(n+1)},v_K^{(n+1)}) \ud \x$ and $g_K^{(n+1)}=\frac{1}{|K|}\int_K g(u_K^{(n+1)},v_K^{(n+1)}) \ud \x, \; \forall K \in \mesh$. 

Now, let us discuss the convergence of the scheme \eqref{rm-hmm}. Assume the existence of $\theta>0$ such that,
for any $m\in\NN$,
\begin{equation}\label{reg.HMM.1}
\begin{aligned}
{\theta_\mesh:=}\max_{K\in\mesh_m}&\left(\dsp\max_{\edge\in\edges_K}\frac{h_K}{d_{K,\edge}}+{\rm Card}(\edges_K)\right)
+\max_{\stackrel{\mbox{\scriptsize $\edge\in \edges_{m,\rm int}$}}{\mesh_\edge=\{K,L\}}}\left( \frac{d_{K,\edge}}{d_{L,\edge}}+\frac{d_{L,\edge}}{d_{K,\edge}} \right)\leq \theta
\end{aligned}\end{equation}
and, for all $K\in \mesh_m$ and $\gamma \in \RR^{\edges_K}$,
\begin{equation}\label{reg.HMM.2}
\begin{aligned}
\frac{1}{\theta}\dsp\sum_{\sigma\in\edges_K}|D_{K,\sigma}| \left| \frac{R_{K,\sigma}(\gamma)}{d_{K,\sigma}} \right|^2
\leq{}& \dsp\sum_{\sigma\in\edges_K} |D_{K,\sigma}| \left| \frac{(A_K R_K(\gamma))_\sigma}{d_{K,\sigma}} \right|^2\\
\leq{}& \theta \dsp\sum_{\sigma\in\edges_K} |D_{K,\sigma}| \left| \frac{R_{K,\sigma}(\gamma)}{d_{K,\sigma}} \right|^2.
\end{aligned}
\end{equation}

Under the above boundedness assumptions on the mesh regularity parameter $\theta_\mesh$, the limit-conformity and the compactness properties defined in Section \ref{sec-disc-pblm} directly follow from the results in \cite[Theorem 13.18]{30}. It remains to prove the consistency. Proving that, for all $\bar\varphi \in H^1(\O)$, $\lim_{m \to \infty}S_{\disc_m}(\bar\varphi)=0$ is as in the case of the
HMM method for PDEs, see \cite[Theorem 4.14]{30}. Also, with the same manner, we can show that, for all $\bar\psi \in L^2(\O)$, $\widehat S_{\disc_m}(\bar\psi) \to 0$, as $m \to \infty$. Let $\varphi_m=( (\varphi_K)_{K\in\mesh}, (\varphi_\sigma)_{\sigma\in\edges} ) \in X_{\disc_m}$ and $\psi_m=(\psi_K)_{K\in\mesh}\in Y_{\disc_m}$ be the two interplants such that $\varphi_m=J_{\disc_m}u_{\rm{ini}}$ and $\psi_m=J_{\disc_m^{'}}v_{\rm{ini}}$. Using \cite[Estimate (B.11), in Lemma B.6]{30} with $p=2$, we can shows that $|| u_{\rm ini}-\Pi_{\disc_m}J_{\disc_m} u_{\rm ini} ||_{L^2(\O)} \to 0$ and $|| v_{\rm ini}-\Pi_{\disc_m^{'}}J_{\disc_m^{'}} v_{\rm ini} ||_{L^2(\O)} \to 0$, as $m \to \infty$, which completes the consistency property.

It is proved in \cite[Theorem 13.14]{30} that, for $w \in W^{1,p}(\O)$, there exists an interpolant $\varphi_m=( (\varphi_K)_{K\in\mesh}, (\varphi_\sigma)_{\sigma\in\edges} ) \in X_{\disc_m}$, such that there is a constant $C_1>0$ not depending on $m$, and 
\[
|| \nabla_{\disc_m}\varphi_m ||_{L^p(\O)^d} \leq C_1 || \nabla w ||_{L^p(\O)^d}.
\]
The boundedness of $|| \nabla_{\disc_m} J_{\disc_m} u_{\rm ini}||_{L^p(\O)^d}$ is a consequence of applying the above estimate to $w=u_{\rm ini}$ (with $p=2$). The convergence of the HMM scheme for the reaction diffusion model is a consequence of Theorem \ref{theorem-conver-rm} if the discrete time steps $\delta t_{\disc_m}^{(n+\frac{1}{2})}$ tends to 0, as $m \to \infty$.

\section{Proof of the convergence results}\label{sec-proof}
In order to prove the convergence results, we first establish some preliminaries estimates on the solution and its gradient of the scheme.

\begin{lemma}[Estimates]
\label{lemma-est-rm}
Under Assumptions \eqref{assump-rm}, let $\disc$ be a gradient discretisation in the sense of Definition \ref{def-gd-rm} and let $(u,v) \in X_\disc \times Y_\disc$ be a solution of the gradient scheme \eqref{rm-disc-pblm}. Then there exists a constant $C_2\geq 0$ depending only on $C_{\rm ini} > \max(|| \Pi_\disc u^{(0)} ||_{L^2(\O)}, || \Pi_{\disc^{'}} v^{(0)} ||_{L^2(\O)})$, such that
\begin{equation}\label{est-rm}
||\Pi_\disc u ||_{L^\infty(0,T;L^2(\O))}
+||\Pi_{\disc^{'}} v ||_{L^\infty(0,T;L^2(\O))} 
+||\nabla_\disc u ||_{ L^2(\O \times (0,T))^d }
\leq C_2 .
\end{equation}
\end{lemma}
\begin{proof}
Setting, as test functions in Scheme \eqref{rm-disc-pblm}, the functions $\varphi:=\delta t^{ (n+\frac{1}{2}) }u^{(n+1)}$ and $\psi:=\delta t^{ (n+\frac{1}{2}) }v^{(n+1)}$ leads to
\begin{equation*}
\begin{aligned}
\dsp\int_\O \Big(\Pi_\disc u^{(n+1)}(\x)&-\Pi_\disc u^{(n)}(\x)\Big) \Pi_\disc u^{(n+1)}(\x) \ud \x
+\mu\dsp\dsp\int_{t^{(n)}}^{t^{(n+1)}}\int_\O |\nabla_\disc u^{(n+1)}(\x)|^2 \ud \x\\
&\quad= \dsp\int_{t^{(n)}}^{t^{(n+1)}}\int_\O f_\disc^{(n+1)} \Pi_\disc u^{(n+1)}(\x)  \ud \x \ud t,
\end{aligned}
\end{equation*}
and
\begin{equation*}
\dsp\int_\O \Big(\Pi_{\disc^{'}} v^{(n+1)}(\x)-\Pi_{\disc^{'}} v^{(n)}(\x)\Big) \Pi_{\disc^{'}} v^{(n+1)}(\x) \ud \x
= \dsp\int_{t^{(n)}}^{t^{(n+1)}}\int_\O g_\disc^{(n+1)} \Pi_\disc v^{(n+1)}(\x)  \ud \x \ud t.
\end{equation*}
Apply the inequality $(y-z)y \geq \frac{1}{2}( |y|^2 -|z|^2 )$ to $y=\Pi_\disc u^{(n+1)}$ and $z=\Pi_\disc u^{(n)}$ (resp. $y=\Pi_{\disc^{'}} v^{(n+1)}$ and $z=\Pi_{\disc^{'}} v^{(n)}$) to obtain
\begin{equation*}
\begin{aligned}
\frac{1}{2}\dsp\int_\O \Big[ |\Pi_\disc u^{(n+1)}(\x)|^2-|\Pi_\disc u^{(n)}(\x)|^2 \Big] \ud \x
&+\mu\dsp\dsp\int_{t^{(n)}}^{t^{(n+1)}}\int_\O |\nabla_\disc u^{(n+1)}(\x)|^2 \ud \x\\
&\quad\leq \dsp\int_{t^{(n)}}^{t^{(n+1)}}\int_\O f_\disc^{(n+1)} \Pi_\disc u^{(n+1)}(\x)  \ud \x \ud t,
\end{aligned}
\end{equation*}
and
\begin{equation*}
\frac{1}{2}\dsp\int_\O \Big[ |\Pi_{\disc^{'}} v^{(n+1)}(\x)|^2-|\Pi_{\disc^{'}} v^{(n)}(\x)|^2 \Big] \ud \x
\leq \dsp\int_{t^{(n)}}^{t^{(n+1)}}\int_\O g_\disc^{(n+1)} \Pi_{\disc^{'}} v^{(n+1)}(\x)  \ud \x \ud t.
\end{equation*}
Sum on $n=0,...,m-1$, for some $m=0,...,N$:
\begin{equation*}
\begin{aligned}
\frac{1}{2}\dsp\int_\O \Big[ |\Pi_\disc u^{(m)}(\x)|^2-|\Pi_\disc u^{(0)}(\x)|^2 \Big] \ud \x
&+\mu\dsp\dsp\int_{0}^{t^{(m)}}\int_\O |\nabla_\disc u(\x)|^2 \ud \x\\
&\quad\leq \dsp\int_{0}^{t^{(m)}}\int_\O f_\disc^{(m)} \Pi_\disc u^{(m)}(\x)  \ud \x \ud t,
\end{aligned}
\end{equation*}
and
\begin{equation*}
\frac{1}{2}\dsp\int_\O \Big[ |\Pi_{\disc^{'}} v^{(m)}(\x)|^2-|\Pi_{\disc^{'}} v^{(0)}(\x)|^2 \Big] \ud \x
\leq \dsp\int_{0}^{t^{(m)}}\int_\O g_\disc^{(m)} \Pi_{\disc^{'}} v^{(m)}(\x)  \ud \x \ud t.
\end{equation*}
Applying the Cauchy--Schwarz inequality to the right--hand side of both inequalities leads to 
\begin{equation*}
\begin{aligned}
\frac{1}{2}\dsp\int_\O \Big[ |\Pi_\disc u^{(m)}(\x)|^2&-|\Pi_\disc u^{(0)}(\x)|^2 \Big] \ud \x
+\mu\dsp\dsp\int_{0}^{t^{(m)}}\int_\O |\nabla_\disc u(\x)|^2 \ud \x \ud t\\
&\quad\leq || f_\disc^{(m)} ||_{L^2(\O\times (0,T))} \; || \Pi_\disc u^{(m)} ||_{L^2(\O\times (0,T))},
\end{aligned}
\end{equation*}
and
\begin{equation*}
\frac{1}{2}\dsp\int_\O \Big[ |\Pi_{\disc^{'}} v^{(m)}(\x)|^2-|\Pi_{\disc^{'}} v^{(0)}(\x)|^2 \Big] \ud \x
\leq || g_\disc^{(m)} ||_{L^2(\O\times (0,T))} 
\; || \Pi_{\disc^{'}} v^{(m)} ||_{L^2(\O\times (0,T))}.
\end{equation*}
Thanks to the assumptions on $f$ and $g$ given in Assumptions \ref{assump-rm}, one writes
\begin{equation*}
\begin{aligned}
\frac{1}{2}\dsp\int_\O \Big[ |\Pi_\disc u^{(m)}(\x)|^2&-|\Pi_\disc u^{(0)}(\x)|^2 \Big] \ud \x
+\mu\dsp\int_{0}^{t^{(m)}}\int_\O |\nabla_\disc u(\x,t)|^2 \ud \x \ud t\\ 
&\leq c_1 || \Pi_\disc u^{(m)} ||_{L^2(\O\times (0,T))} + c_2 || \Pi_\disc u^{(m)} ||_{L^2(\O\times (0,T))}^2\\
&+ c_3 | \Pi_\disc u^{(m)} ||_{L^2(\O\times (0,T))} || \Pi_{\disc^{'}}  v^{(m)} ||_{L^2(\O\times (0,T))},
\end{aligned}
\end{equation*}
and
\begin{equation*}
\begin{aligned}
\frac{1}{2}\dsp\int_\O \Big[ |\Pi_{\disc^{'}} v^{(m)}(\x)|^2&-|\Pi_{\disc^{'}} v^{(0)}(\x)|^2 \Big] \ud \x\\
&\leq c_4 || \Pi_{\disc^{'}} v^{(m)} ||_{L^2(\O\times (0,T))} 
+c_6 || \Pi_{\disc^{'}} v^{(m)} ||_{L^2(\O\times (0,T))}^2\\ 
&+c_5 || \Pi_{\disc^{'}} v^{(m)} ||_{L^2(\O\times (0,T))} || \Pi_\disc u^{(m)} ||_{L^2(\O\times (0,T))}.
\end{aligned}
\end{equation*}
Apply the Young's inequality, with $\varepsilon$ satisfying $M_1+\frac{1}{2\varepsilon}-\frac{1}{2} >0$ and $M_2+\frac{1}{2\varepsilon}-\frac{1}{2} >0$, to the right--hand side of both inequalities to obtain
\begin{equation*}
\begin{aligned}
\frac{1}{2}\dsp\int_\O \Big[ |\Pi_\disc u^{(m)}(\x)|^2&-|\Pi_\disc u^{(0)}(\x)|^2 \Big] \ud \x
+\mu\dsp\dsp\int_{0}^{t^{(m)}}\int_\O |\nabla_\disc u(\x,t)|^2 \ud \x \ud t\\
&\quad\leq \frac{c_1^2}{2\varepsilon}+M_1 || \Pi_\disc u^{(m)} ||_{L^2(\O\times (0,T))}^2
+ \frac{1}{2\varepsilon} || \Pi_{\disc^{'}} v^{(m)} ||_{L^2(\O\times (0,T))}^2
\end{aligned}
\end{equation*}
and
\begin{equation*}
\begin{aligned}
\frac{1}{2}\dsp\int_\O \Big[ |\Pi_{\disc^{'}} v^{(m)}(\x)|^2&-|\Pi_{\disc^{'}} v^{(0)}(\x)|^2 \Big] \ud \x\\
&\quad\leq \frac{c_4^2}{2\varepsilon}+M_2 || \Pi_{\disc^{'}} v^{(m)} ||_{L^2(\O\times (0,T))}^2
+ \frac{1}{2\varepsilon} || \Pi_\disc u^{(m)} ||_{L^2(\O\times (0,T))}^2,
\end{aligned}
\end{equation*}
where $M_1:=\frac{\varepsilon}{2}+c_2+\frac{c_3^2}{2}$ and $M_2:=\frac{\varepsilon}{2}+c_6+\frac{1}{2\varepsilon}$. 

Combine the above inequalities together and take the supremum on $m=0,...,N$ establishes Estimate \eqref{est-rm}, since $\sup_{m=0,...,N}\int_\O |\Pi_\disc U^{(m)}|^2 \ud \x
=|| \Pi_\disc U^{(m)} ||_{L^\infty(0,T;L^2(\O))}^2
$ (resp. $\sup_{m=0,...,N}\int_\O |\Pi_{\disc^{'}} U^{(m)}|^2 \ud \x
=|| \Pi_{\disc^{'}} U^{(m)} ||_{L^\infty(0,T;L^2(\O))}^2
$).
\end{proof}


\begin{corollary}
Assume \eqref{assump-rm} and let $\disc$ be a gradient discretisation. Then there exists at least one solution $(u,v)$ to the gradient scheme \eqref{def-gs-rm}.
\end{corollary}

\begin{proof}
At each time step $n+1$, \eqref{rm-disc-pblm} provides square non linear equations on $u^{(n+1)}$ and $v^{(n+1)}$. For a given $w=(w_1,w_2) \in X_\disc \times Y_\disc$, $(u,v)\in X_\disc \times Y_\disc$ is is the solution to
\begin{equation}\label{rm-disc-pblm-lin}
\begin{aligned}
&\dsp\int_\O \Pi_{\disc}\dsp\frac{u^{(n+1)}-u^{(n)}}{\delta t^{(n+\frac{1}{2})}}(x) \Pi_\disc \varphi(\x)
+ \mu\dsp\int_\O \nabla_\disc u^{(n+1)}(\x) \cdot \nabla_\disc \varphi(\x)\ud \x\\
&\qquad=\dsp\int_\O f(\Pi_\disc w_1, \Pi_{\disc^{'}} w_2)\Pi_\disc \varphi(\x) \ud \x, \quad \forall \varphi \in X_\disc, \\
&\dsp\int_\O \Pi_{\disc^{'}}\dsp\frac{v^{(n+1)}-v^{(n)}}{\delta t^{(n+\frac{1}{2})}} \Pi_{\disc^{'}} \psi(\x) = \dsp\int_\O g(\Pi_\disc w_1, \Pi_{\disc^{'}} w_2)\Pi_\disc \psi(\x) \ud \x ,\quad \forall \psi \in Y_\disc. 
\end{aligned}
\end{equation}

This problem describes a linear square system, whose a right hand-side is built from the terms $\int_\O f(\Pi_\disc w_1,\Pi_{\disc^{'}}w_2)\Pi_\disc \varphi \ud \x$, $\int_\O \Pi_\disc u^{(n)}\Pi_\disc \varphi \ud \x$, $\int_\O g(\Pi_\disc w_1,\Pi_{\disc^{'}}w_2)\Pi_{\disc^{'}} \psi \ud \x$, and $\int_\O \Pi_{\disc^{'}} v^{(n)}\Pi_{\disc^{'}} \psi \ud \x$.

Using arguments similar to the proof of Lemma \eqref{lemma-est-rm-grad}, we can obtain
\begin{equation*}
\begin{aligned}
&|| \Pi_\disc u^{(n+1)} ||_{L^2(\O)} + || \nabla_\disc u^{(n+1)} ||_{L^2(\O)^d} \leq
C_4|| f ||_{L^2(\O)} + || \Pi_{\disc} u^{(n)}||_{L^2(\O)},\\
&|| \Pi_{\disc^{'}} v^{(n+1)} ||_{L^2(\O)} \leq
C_5|| g ||_{L^2(\O)} + || \Pi_{\disc^{'}} v^{(n)}||_{L^2(\O)}.
\end{aligned}
\end{equation*}
where $C_4$ and $C_5$ not depending on $u^{(n+1)}$ or $v^{(n+1)}$. Hence, the kernel of the matrix associated with the above linear square system is reduced to $\{0\}$, and the matrix is invertible. We then can define the mapping $T:X_\disc \times Y_\disc \to X_\disc \times Y_\disc$ by $T(w)=(u,v)$ with $(u,v)$ is the solution to \eqref{rm-disc-pblm-lin}. Since $T$ is continuous, Brouwer's fixed point establishes the existence of a solution $u^{(n+1)}$ to the system at time step $n + 1$. 
\end{proof}


\begin{lemma}[Estimates on the discrete gradient and time derivative]
\label{lemma-est-rm-grad}
Under Assumptions \eqref{assump-rm}, let $\disc$ be a gradient discretisation in the sense of Definition \ref{def-gd-rm} and let $(u,v)$ be a solution of the gradient scheme \eqref{rm-disc-pblm}. If $ || \nabla_\disc u^{(0)} ||_{L^2(\O)} $ is  bounded by $C_6$, then there exists a constant $C_7$ only depending on $C_2$ and $C_6$, such that
\begin{equation}\label{eq-est-grad}
||  \delta_\disc u||_{L^2(\O\times (0,T))}+ || \nabla_\disc u ||_{L^\infty(0,T;L^2(\O)^d)} \leq C_7
\end{equation}
\end{lemma}

\begin{proof}
Setting, as a test function in Scheme \eqref{rm-disc-pblm}, the function $ \varphi=u^{(n+1)}-u^{(n)}$ leads to
\begin{equation*}
\begin{aligned}
\delta t^{(n+\frac{1}{2})}\dsp\int_\O |\delta_\disc^{(n+\frac{1}{2})} u(\x) |^2 \ud \x
&+\mu\dsp\int_\O \nabla_\disc u^{(n+1)}(\x) \cdot \nabla_\disc(u^{(n+1)}(\x)-u^{(n)}(\x)) \ud \x \\
&\quad= \dsp\int_{t^{(n)}}^{t^{(n+1)}}\int_\O f_\disc^{(n+1)}\delta_\disc^{(n+\frac{1}{2})} u(\x)  \ud \x \ud t.
\end{aligned}
\end{equation*}
Applying the relation $(r - q) \cdot r \geq  \frac{1}{2} |r|^2 - \frac{1}{2} |q|^2$ to the second term on the left hand side yields
\begin{equation*}
\begin{aligned}
\delta t^{(n+\frac{1}{2})}\dsp\int_\O |\delta_\disc^{(n+\frac{1}{2})} u(\x) |^2 \ud \x
&+\frac{\mu}{2}\dsp\int_\O |\nabla_\disc u^{(n+1)}(\x)|^2 \ud \x 
-\frac{\mu}{2}\dsp\int_\O |\nabla_\disc u^{(n)}(\x)|^2 \ud \x \\
&\quad\leq
\dsp\int_{t^{(n)}}^{t^{(n+1)}}\int_\O f_\disc^{(n+1)} \delta_\disc^{(n+\frac{1}{2})} u(\x)  \ud \x \ud t.
\end{aligned}
\end{equation*}
Sum this inequality on $n=0,...,m-1$, for some $m=0,...,N$ to get
\begin{equation*}
\begin{aligned}
\dsp\int_0^{t^{(m)}}\dsp\int_\O |\delta_\disc u(t) |^2 \ud \x \ud t
&+\frac{\mu}{2}\dsp\int_\O |\nabla_\disc u^{(m)}(\x)|^2 \ud \x 
-\frac{\mu}{2}\dsp\int_\O |\nabla_\disc u^{(0)}|^2 \ud \x \\
&\quad\leq
\dsp\int_0^{t^{(m)}}\int_\O f_\disc^{(m)} \delta_\disc u(t)  \ud \x \ud t.
\end{aligned}
\end{equation*}
Thanks to Assumptions \ref{assump-rm}, Cauchy--Schwars inequality and Young's inequalities, we have
\begin{equation*}
\begin{aligned}
\dsp\int_0^{t^{(m)}}\dsp\int_\O |\delta_\disc u(t) |^2 \ud \x \ud t
&+\frac{\mu}{2}\dsp\int_\O |\nabla_\disc u^{(m)}(\x)|^2 \ud \x 
 \\
&\quad\leq
\frac{1}{2}\Big( c_1+c_2 || \Pi_\disc u ||_{L^2(\O\times (0,t^{(m)}))} + c_3 || \Pi_\disc v ||_{L^2(\O\times (0,t^{(m)}))} \Big)^2\\
&\qquad+\frac{1}{2}|| \delta_\disc u ||_{L^2(\O\times (0,t^{(m)}))}^2
+\frac{\mu}{2}C_6.
\end{aligned}
\end{equation*}
The proof is concluded by taking the supremum on $m=0,...,N$, using the inequality $\sup_n(a_n+b_n) \leq \sup_n(a_n)+\sup{_n}(b_n)$ and invoking Lemma \eqref{lemma-est-rm} to estimate the second and third terms on the right hand side.   
\end{proof}

\subsection*{Proof of Theorem \ref{theorem-conver-rm}}

The proof follows the compactness technique detailed in \cite{44} and it is split into three steps:
\vskip 1pc
\noindent {\bf Step 1:} Compactness results.\\
Due to Estimate \eqref{est-rm} and to the consistency and the limit--conformity of GD, \cite[Lemma 4.8]{30} provides $\bar u \in L^2(0,T;H^1(\O))$ and $\bar v \in L^2(0,T;L^2(\O))$, such that, up to a subsequence, $\Pi_{\disc_m}u_m \to \bar u$ weakly in $L^2(0,T,L^2(\O))$, $\Pi_{\disc_m^{'}}v_m \to \bar v$ weakly in $L^2(0,T,L^2(\O))$, and $\nabla_{\disc_m}u_m \to \nabla\bar u$ weakly in $L^2(0,T,L^2(\O)^d)$, as $m\to \infty$. Since $ || \nabla_\disc u^{(0)} ||_{L^2(\O)}$ is assumed to bounded, Estimate \eqref{eq-est-grad} holds. This estimate with the three properties ( consistency, limit--conformity and compactness) show that the assumptions of \cite[Theorem 4.18]{30} are satisfied. This theorem proves that $\Pi_{\disc_m}u_m$ converges strongly to $\bar u$ in $L^\infty(0,T;L^2(\O))$, as $m\to \infty$. In  fact, $\partial_t \bar u \in L^2(0,T;L^2(\O)) $ and $\delta_{\disc_m}u_m$ converges weakly to $\partial_t\bar u$ in $L^2(0,T;L^2(\O))$, as $m\to \infty$. Estimates \eqref{est-rm} yields the weak convergence of $\Pi_{\disc_m}v_m$ to $\bar v$ in $L^\infty(0,T;L^2(\O))$, as $m\to \infty$.  

\vskip 1pc
\noindent {\bf Step 2:} $(\bar u,\bar v)$ is a solution to the continuous problem.\\
The a.e. convergence of $\Pi_{\disc_m}u_m$, the assumptions on $f_1$ and $f_2$ and the dominated convergence theorem show that $f_1(\Pi_{\disc_m}u_m) \to f_1(\bar u)$ and $f_2(\Pi_{\disc_m}u_m) \to f_2(\bar u)$ in $L^2(\O\times(0,T))$. Consider $\bar\varphi \in L^2(0,T;H^1(\O))$. \cite[Lemma 4.10]{30} shows the existence of $w_m=(w_m^{(n)})_{n=0,...,N_m} \in X_{\disc_m}^{N_m+1}$, such that $\Pi_{\disc_m}w_m \to \bar\varphi$ in $L^2(0,T;L^2(\O))$ and $\nabla_{\disc_m}w_m \to \nabla\bar\varphi$ in $L^2(0,T;L^2(\O)^d)$. Take $\varphi=\delta t_m^{(n+\frac{1}{2})}w_m^{(n)}$ as a test function in \eqref{rm-disc-pblm1}, sum on $n=0,...,N_m-1$, and pass to the limit $m\to \infty$ to see that $(\bar u,\bar v)$ satisfies \eqref{rm-weak1}, thanks to the strong--weak convergence of $\Pi_{\disc_m}u_m$, $\Pi_{\disc_m^{,}}v_m$ and $\delta_{\disc_m}u_m$.
\par Let us now verify \eqref{rm-weak2}. Take $\bar\psi \in L^2(0,T;L^2(\O))$ such that $\partial_t \bar \psi \in L^2(\O\times(0,T))$ and $\bar\psi(T,\cdot)=0$. By \cite[Lemma 4.10]{30} (with a slight change), we can find $w_m=(w_m^{(n)})_{n=0,...,N_m} \in Y_{\disc_m}^{N_m+1}$, such that $\Pi_{\disc_m^{'}}w_m \to \bar\psi$ in $L^2(0,T;L^2(\O))$ and $\delta_{\disc_m^{'}}w_m \to \partial_t \bar\psi$ strongly in $L^2(\O\times(0,T))$. Set $\psi=\delta t_m ^{(n+\frac{1}{2})}w_m^{(n)}$ as a test function in \eqref{rm-disc-pblm2} and sum on $n=0,...,N_m-1$ to obtain
\begin{equation}\label{eq-part1}
\begin{aligned}
\dsp\sum_{n=0}^{N_m-1}\dsp\int_\O \left( \Pi_{\disc_m^{'}}v_m^{(n+1)}(\x)-\Pi_{\disc_m^{'}}v^{(n)}(\x) \right)&\Pi_{\disc_m^{'}}w_m^{(n)}(\x) \ud x\\
&=\dsp\int_0^T\dsp\int_\O g_{\disc_m}^{(n+1)} \Pi_{\disc_m^{'}}w_m^{(n)}(\x) \ud x \ud t.
\end{aligned}
\end{equation}
For two families of real numbers $(a_n)_{n=0,...,\NN}$ and $(b_n)_{n=0,...,N}$, the form of discrete integration by part \cite[Eq. (D.15)]{30} is
\[
\dsp\sum_{n=0}^{N-1}\Big(b_{n+1}-b_n\Big)a_n=b_N a_N
-b_0 a_0-\dsp\sum_{n=0}^{N-1}b_{n+1}\Big(a_{n+1}-a_n\Big).
\]
Apply this relation to the right hand side in \eqref{eq-part1} and use the fact $w^{(N)}=0$ to deduce 
\begin{equation*}
\begin{aligned}
-\dsp\int_0^T \dsp\int_\O \Pi_{\disc_m^{'}}v_m(\x,t)\delta_{\disc_m^{'}}w_m(\x,t) \ud \x \ud t 
&-\int_\O \Pi_{\disc_m^{'}}v_m^{(0)}(\x)\Pi_{\disc_m^{'}}w_m^{(0)}(\x) \ud \x\\
&\quad=\dsp\int_0^T\dsp\int_\O g_{\disc_m}^{(n+1)} \Pi_{\disc_m^{'}}w_m^{(n)}(\x) \ud x \ud t.
\end{aligned}
\end{equation*}
By the space--time consistency, $\Pi_{\disc_m^{'}}v_m^{(0)}=\Pi_{\disc_m^{'}}J_{\disc_m^{'}}v_{\rm ini} \to v_{\rm ini}$ in $L^2(\O)$. The strong convergence of $\Pi_{\disc_m}u_m$, the assumptions on $g_1$ and the dominated convergence theorem lead to $g_1(\Pi_{\disc_m}u_m) \to g_1(\bar u)$ in $L^2(\O\times(0,T))$. Hence, passing to the limit $m \to \infty$ in the above equation implies that $\bar v$ satisfies \eqref{rm-disc-pblm2}.
\vskip 1pc
\noindent {\bf Step 3:} $\nabla_{\disc_m}u_m$ converges strongly.\\
Taking $\bar\varphi:=u_m$ in \eqref{rm-disc-pblm1} and passing to the superior limit gives (choose $\varphi = \bar u$ in \eqref{rm-weak1})
\begin{equation}\label{eq-proof-grad}
\begin{aligned}
\dsp\lim\sup_{m\to \infty}\dsp\int_0^T\dsp\int_\O \nabla_{\disc_m}&u_m(\x,t)\cdot \nabla_{\disc_m}u_m(\x,t) \ud x \ud t\\
&=\dsp\int_0^T\dsp\int_\O f(\bar u,\bar v)\bar u(\x,t) \ud \x \ud t
-\dsp\int_0^T\dsp\int_\O \partial_t \bar u(\x,t) \bar u(\x,t) \ud \x \ud t\\
&=\dsp\int_0^T\dsp\int_\O \nabla\bar u(\x,t)\cdot \nabla\bar u(\x,t)\ud \x \ud t.
\end{aligned}
\end{equation}
One can write
\begin{equation*}
\begin{aligned}
&\dsp\int_0^T\dsp\int_\O \left(\nabla_{\disc_m}u_m(\x,t)- \nabla\bar u(\x,t)\right)\cdot(\nabla_{\disc_m}u_m(\x,t)- \nabla\bar u(\x,t)) \ud \x \ud t\\
&\quad=\dsp\int_0^T\dsp\int_\O \nabla_{\disc_m}u_m(\x,t) \cdot \nabla_{\disc_m}u_m(\x,t) \ud \x \ud t
-\dsp\int_0^T\dsp\int_\O \nabla_{\disc_m}u_m(\x,t) \cdot \nabla\bar u(\x,t) \ud \x \ud t\\
&\qquad-\dsp\int_0^T\dsp\int_\O \nabla\bar u(\x,t) \cdot (\nabla_{\disc_m}u_m(\x,t)-\nabla\bar u(\x,t)) \ud \x \ud t.
\end{aligned}
\end{equation*}
Thanks to the weak convergence of $\nabla_{\disc_m} u_m$ and \eqref{eq-proof-grad}, passing to the limit $m \to \infty$ in each of the terms above completes the proof.

\section{Numerical results}\label{sec-numerical}
We examine here the validity of the HMM method in solving the reaction-diffusion models by exploring the dynamics of the spiral waves of the Barkley model. This model is considered as an example of the reaction-diffusion model \eqref{rm-strong1}--\eqref{rm-strong5} with reaction terms taking the form \cite{10}:
\begin{equation}\label{eqBK}
\begin{aligned}
 f(\bar u,\bar v)&=\frac{1}{\rho}\bar u(1-\bar u)(\bar u-\frac{\bar v+b}{a}), \\
 g(\bar u,\bar v)&=\bar u-\bar v,
\end{aligned}
\end{equation}
for $a,b>0$. The small parameter $0<\rho\ll1$ represents the time scale separation of fast variable $\bar u$ and slow variable $\bar v$.
 Before proceeding to do that, we first review the dynamics of the underlying model to understand the mechanism of nucleation and propagation of spiral waves in excitable media.


\subsection{Dynamics of the Barkley model}
A pair of essential features are shared by all reaction-diffusion models: spatially localized excitation that undergoes diffusion in space and quickly reverts to a recovery state. Precisely, if the stimulus is strong enough, an excitable system will switch from a state of rest to a state of excitement, then quickly revert to a refractory state and back to rest. A new excitement cannot be created until a certain period (refractory time) has gone by. So an excitable system has the ability to support the wave propagation caused by strong disturbances when in a rest state driven by local nonlinearity coupled with diffusion \cite{9,10}. Figure 1 \cite{43} shows the dynamics of the Barkley model or reaction kinetics when diffusion is not present, with nullclines pictures of $\bar u$ and $\bar v$. The $\bar u$-nullclines ($f(\bar u,\bar v)=0$) are represented by three straight lines: $\bar u=0, \bar u=1, \bar u=\frac{\bar v+b}{a}$, whereas $\bar v$-nullclines ($g(\bar u,\bar v)=0$) is the line $\bar u=\bar v$, where $a,b >0$. Excitable media systems are characterized by dynamic states of excitation and recovery. Thus, by setting a small boundary, say $\delta$, bordering the line $\bar u=0$, a given point $(\bar u,\bar v)$ is said to be excited if $\bar u>\delta$ and recovering otherwise.
\begin{center}
\begin{figure}[!h]
\includegraphics[width=0.80\linewidth]{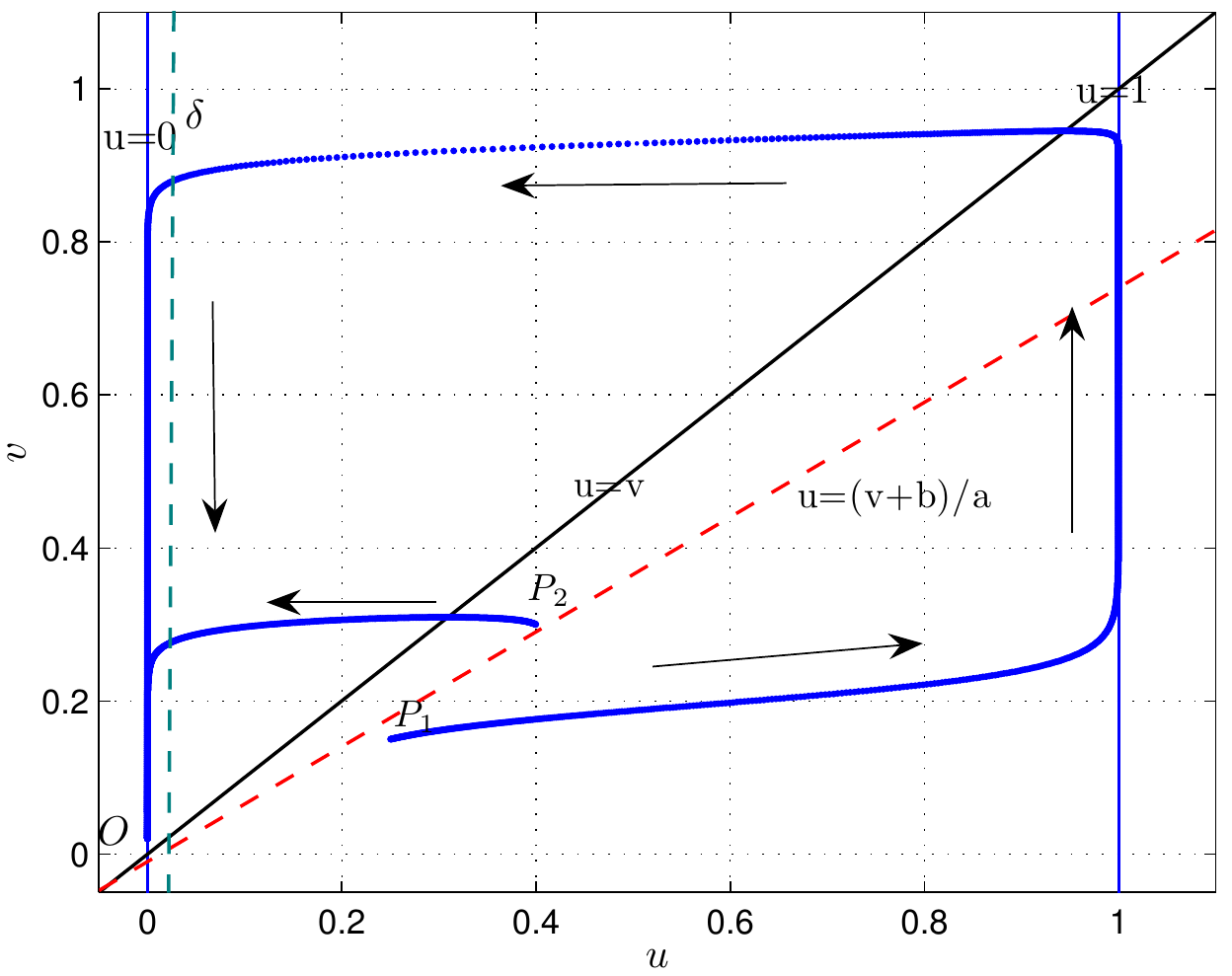}
\caption{Illustration of local dynamics of Barkley model in absence of diffusion ($\mu=0$) taken from \cite{43}. 
}
\label{fig:1}
\end{figure}
\end{center}
 The detail of local dynamics is given by the physical parameters $\rho$, $a$ and $b$. $\rho$ is selected very small, so that the dynamics of the activator $\bar u$ is much faster than the inhibitor $\bar v$ inside the region of excitement. However, $\bar u\approx0$ within recovery region and therefore the exponentially decay of the inhibitor $\bar v$ affects only the local dynamics \cite{10}. 

Increasing $a$ would cause a lengthening of the excitation period and increasing $\frac{b}{a}$ would raise the excitation threshold \cite{25}. Intersection of all nullclines yields the fixed points $(0,0)$ and $(1,1)$. The origin $(0,0)$ is stable and excitable fixed point of the model with excitation threshold $u_{th}=\frac{\bar v+b}{a}$.

To be precise, the initial conditions found on the left of the threshold  $u_{th}$ near the $(0,0)$ decay straight to the fixed point of origin. Conversely, the initial conditions found on the right of $u_{th}$ take a large excursion prior to shrinking down to $(0,0)$, see \cite{10} for more details. When spatial diffusion is added to the reaction kinetics we gain spiral wave dynamics \cite{25}. This allows us to examine the ways in which spiral waves are created, their movement, the way they collapse and disappear, and this offers us insights into the appearance, modification, and ways of stopping arrhythmias.

\subsection{Spiral waves}
\par The wave fronts and backs of spiral waves meet to form a tip which pivots around the region named the core. If the core of the spiral wave moves then it is regarded as being drifting. In order to examine the ways in which spiral waves behave, employing the basis of the HMM method, we consider a pair of examples for the Barkley model at the edge of the simulation domains that have no-flux boundary conditions. In both cases, the HMM scheme is tasted on a square domain $\Omega=[-L,L]^2$, which dividing into a uniform mesh of $3584$ triangular elements.

In all test cases, we consider the Barkley model defined in the introduction of this section \eqref{eqBK}. Letting $\x=(x_1,x_2)$, the initial conditions, required to initiate the spiral waves, are chosen in general form by
\begin{align}
 \bar u_0(x_1,x_2)&=(1+\exp(4(|x_1|-\alpha_1)))^{-2} -(1+\exp(4(|x_1|-\alpha_2)))^{-2},\nonumber \\
 \bar v_0(x_1,x_2)&=\alpha_3,\label{eq2}
\end{align}
where $\alpha_1$,$\alpha_2$ and $\alpha_3$ are real constants. For the time--discretisation, the forward Euler method is considered with time step $\delta t >0$. 
\subsection*{Example 1}\label{test1}
For the first example, the parameters of the model are chosen according to \cite{16} as $\rho=0.005$, $a=0.3$, $b=0.01$, $\mu=1$ and $L=30$. The initial conditions are given as
\[
u_{\rm ini}(x_1,x_2)=\begin{cases}
\hfill 0&, \quad \mbox{for } x_1<0 \mbox{ or }  x_2>5,\\
\bar u_0&, \quad \mbox{otherwise},
\end{cases}
\]
and
\[
v_{\rm ini}(x_1,x_2)=\begin{cases}
\hfill \bar v_0&, \quad \mbox{for } x_1<1 \mbox{ and }  x_2<10,\\
0&, \quad \mbox{otherwise},
\end{cases}
\]
where $\bar u_0$ and $\bar v_0$ are defined in Equation \eqref{eq2}, by setting $\alpha_1=5$, $\alpha_2=1$ and $\alpha_3=0.1$.

\begin{figure}[ht]
	\begin{center}
	\begin{tabular}{ccc}
\includegraphics[width=0.30\linewidth]{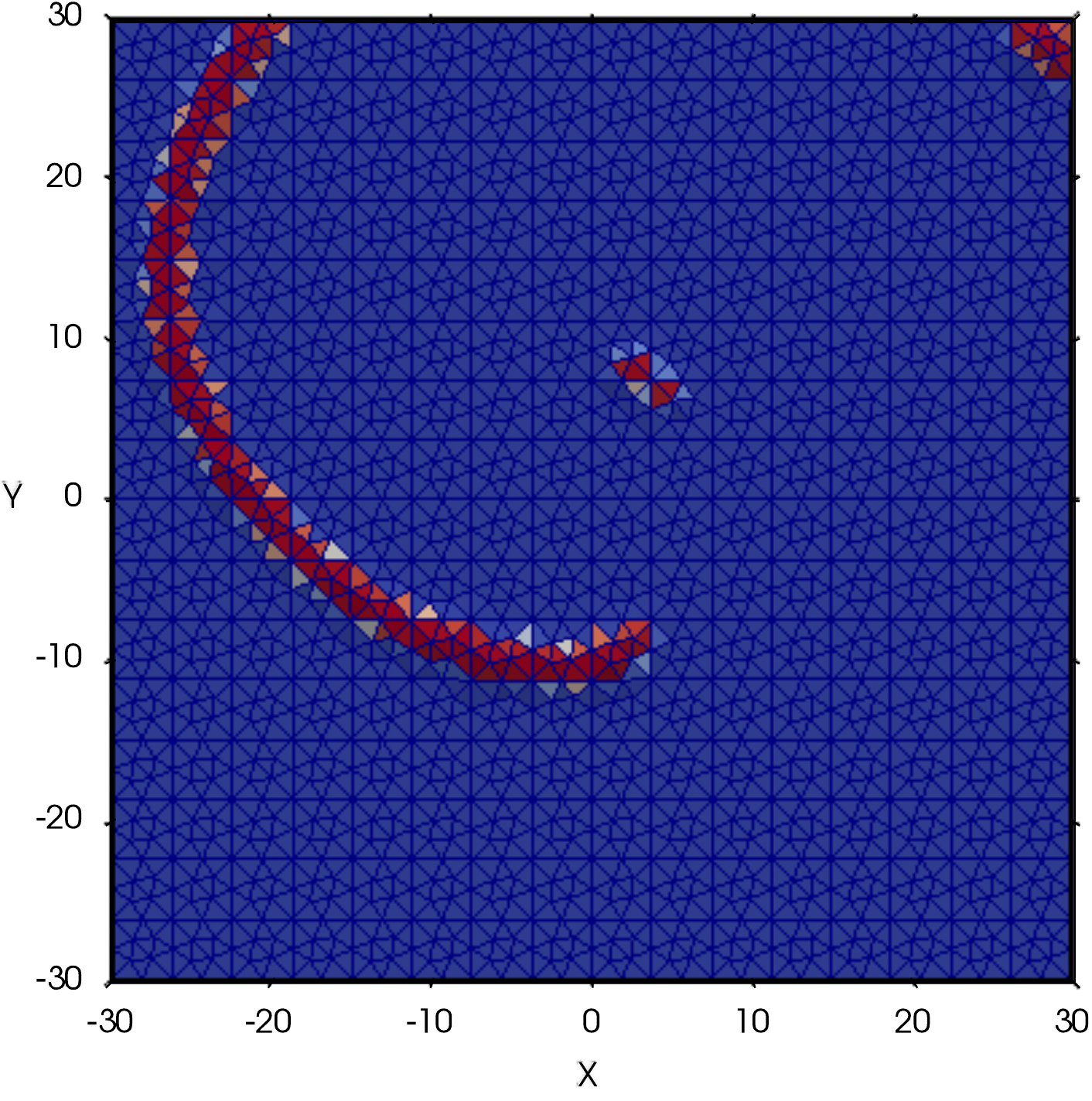} & \includegraphics[width=0.30\linewidth]{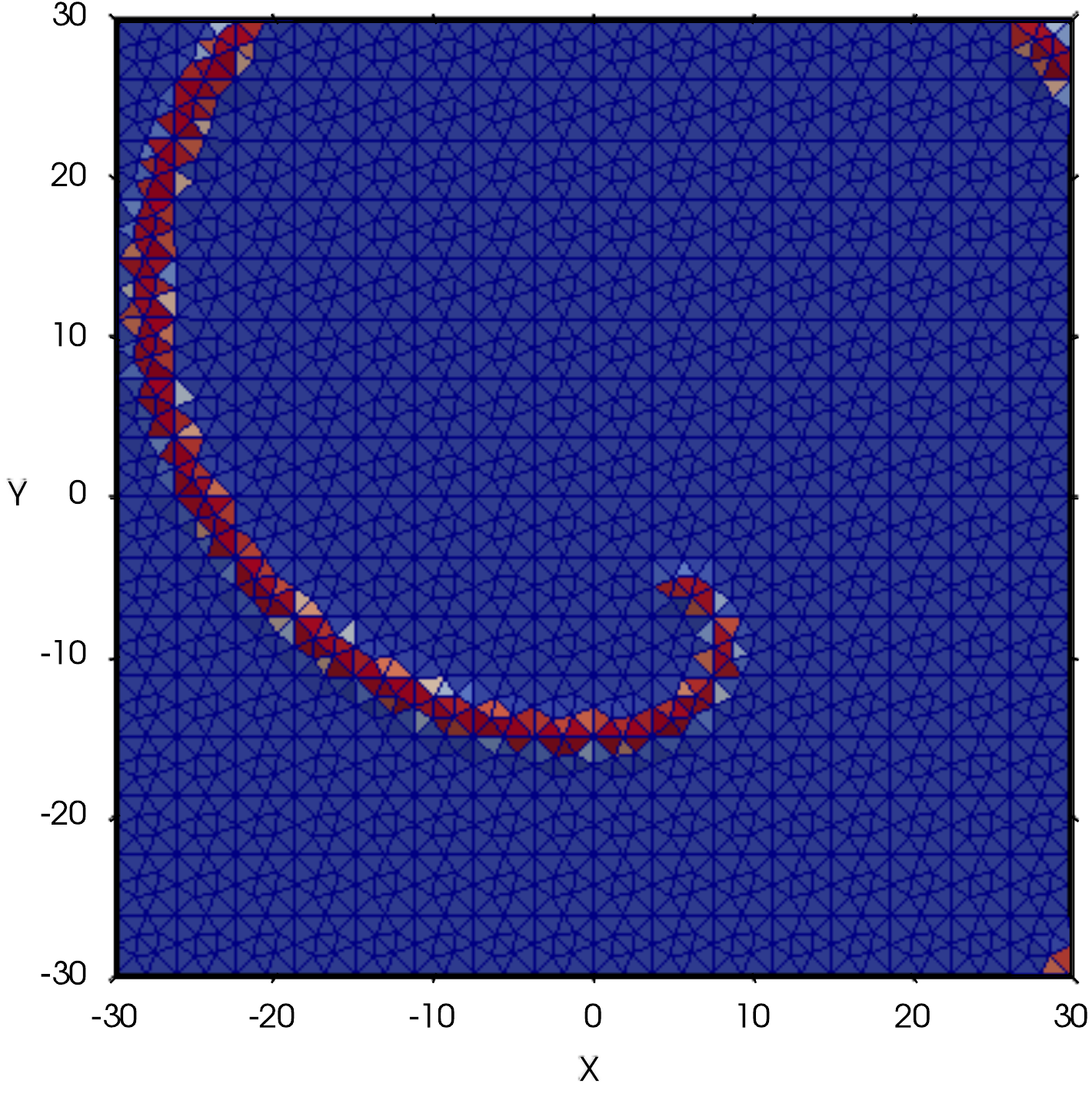} &
\includegraphics[width=0.30\linewidth]{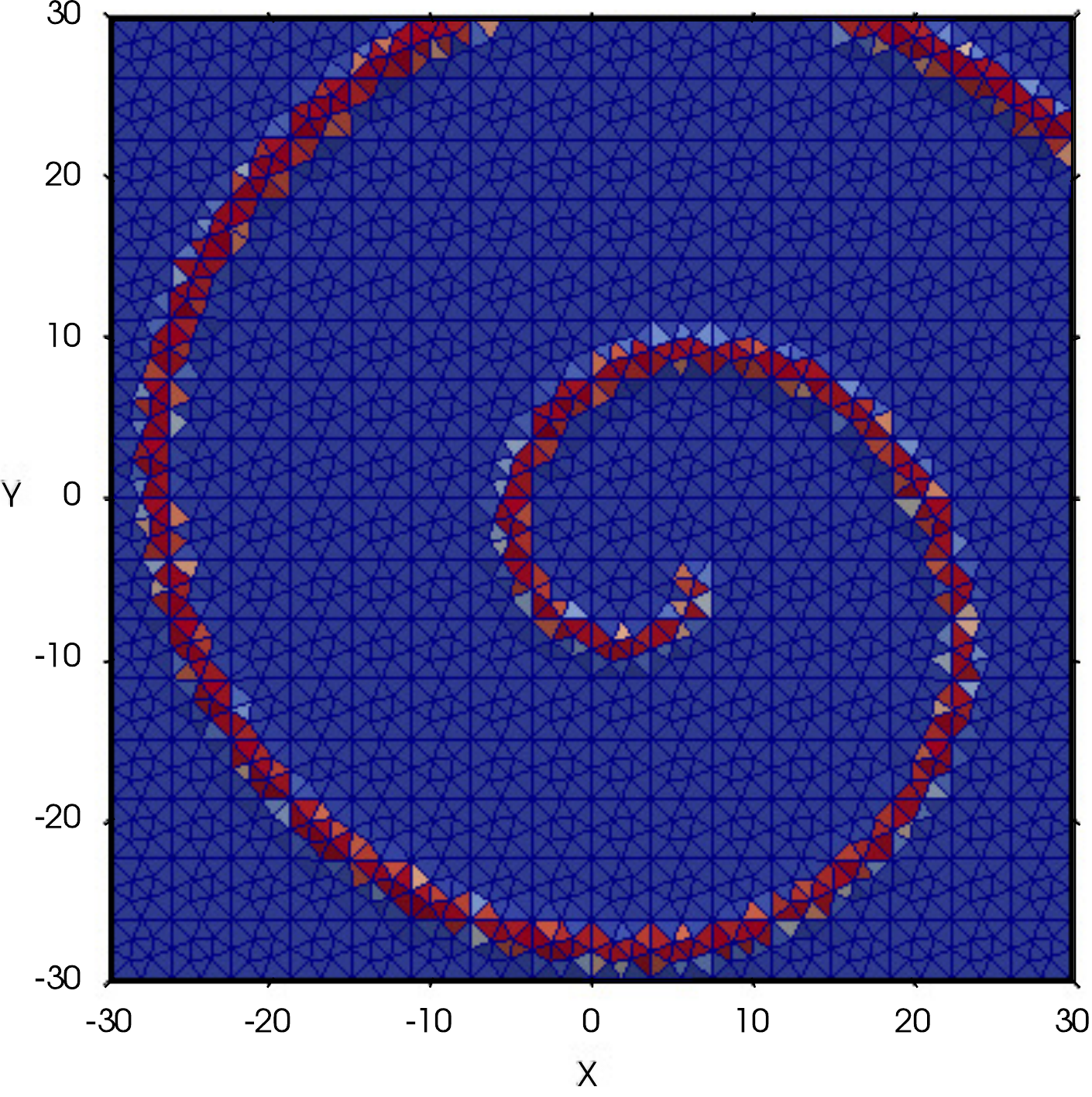}\\
\texttt{Time:} $t=5$ & \texttt{Time:} $t=10$ & \texttt{Time:} $t=20$\\
\includegraphics[width=0.30\linewidth]{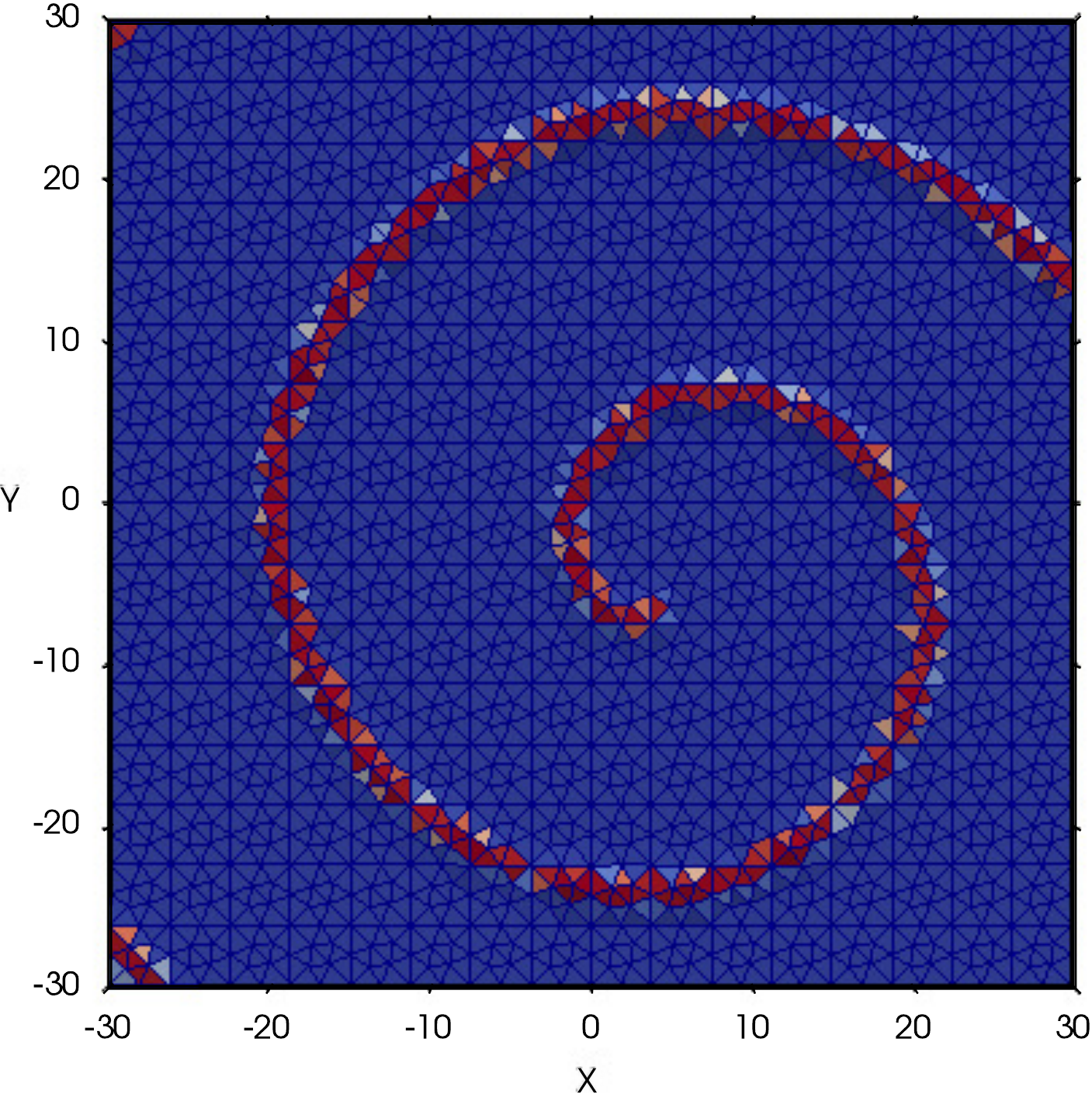} &
\includegraphics[width=0.30\linewidth]{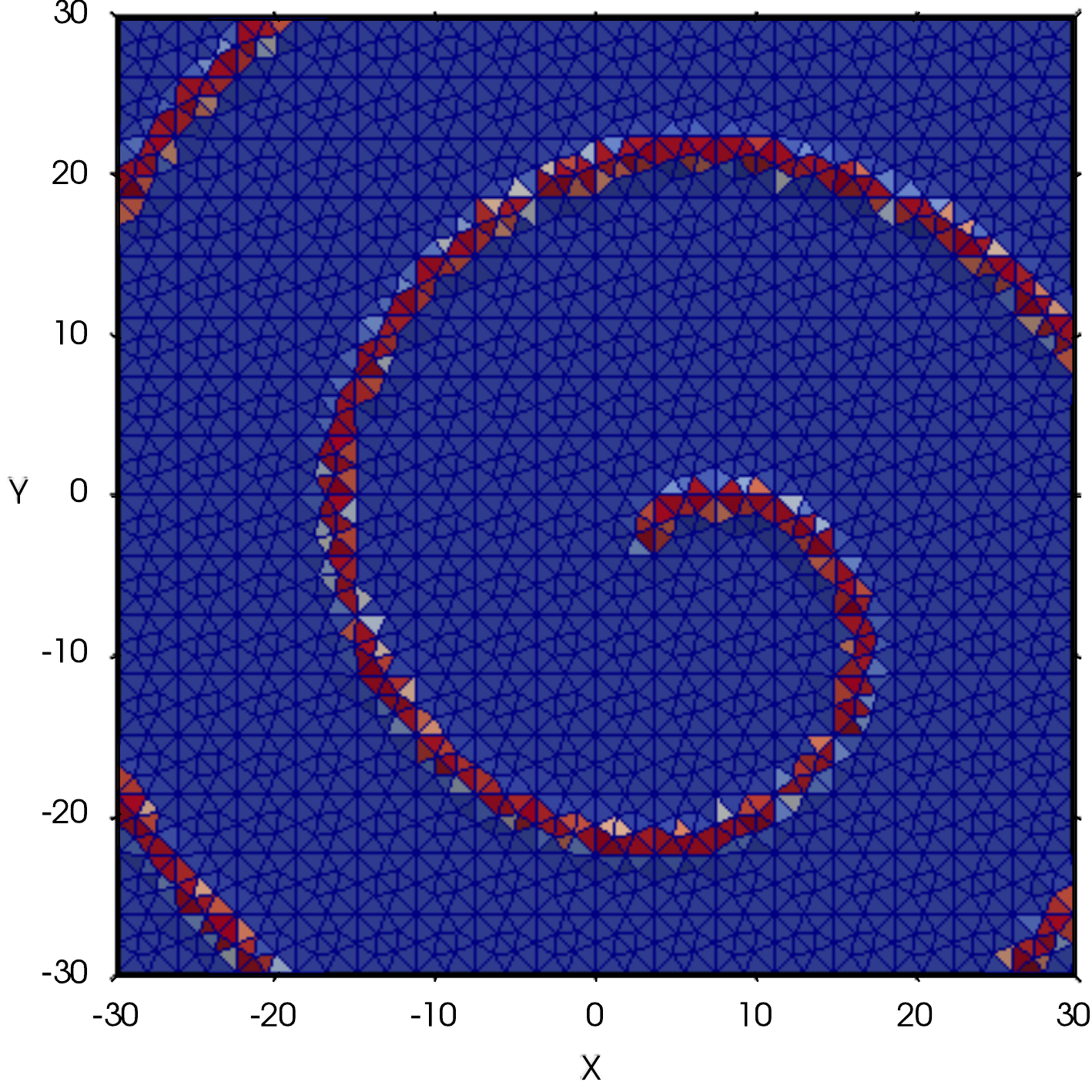} &
\includegraphics[width=0.30\linewidth]{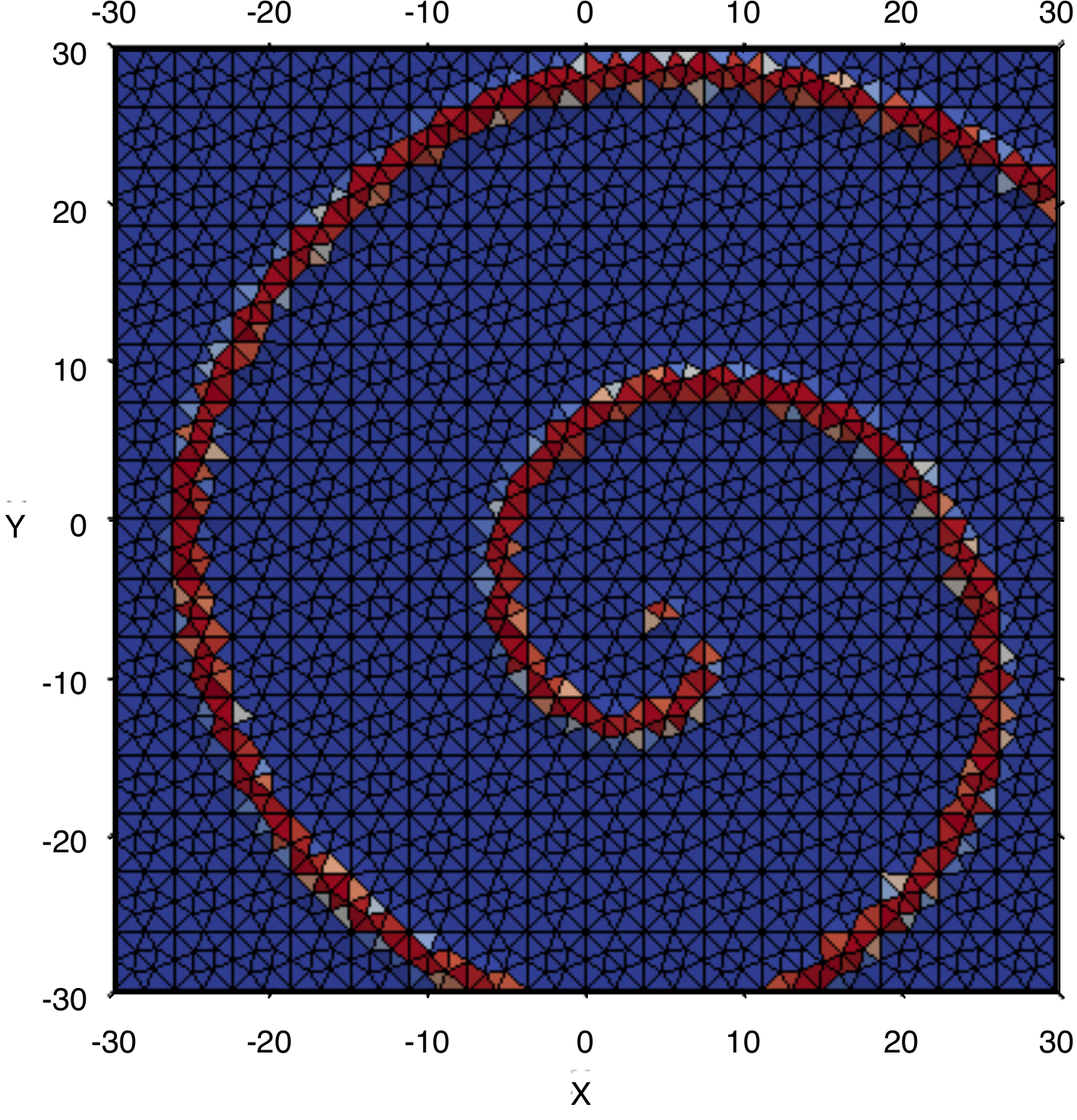}\\
\texttt{Time:} $t=30$ & \texttt{Time:} $t=40$ & \texttt{Time:} $t=60$\\
\includegraphics[width=0.30\linewidth]{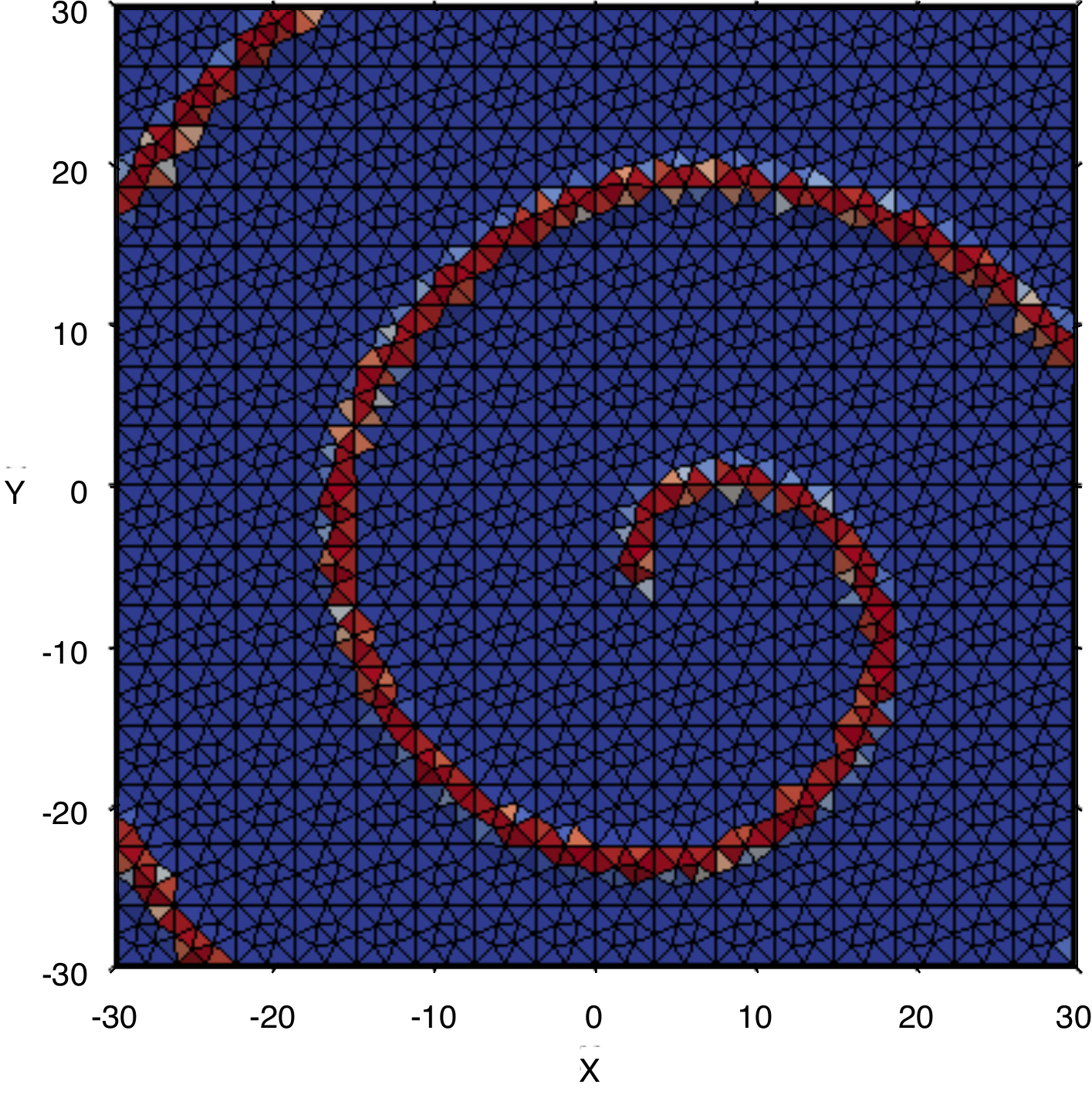} &
\includegraphics[width=0.30\linewidth]{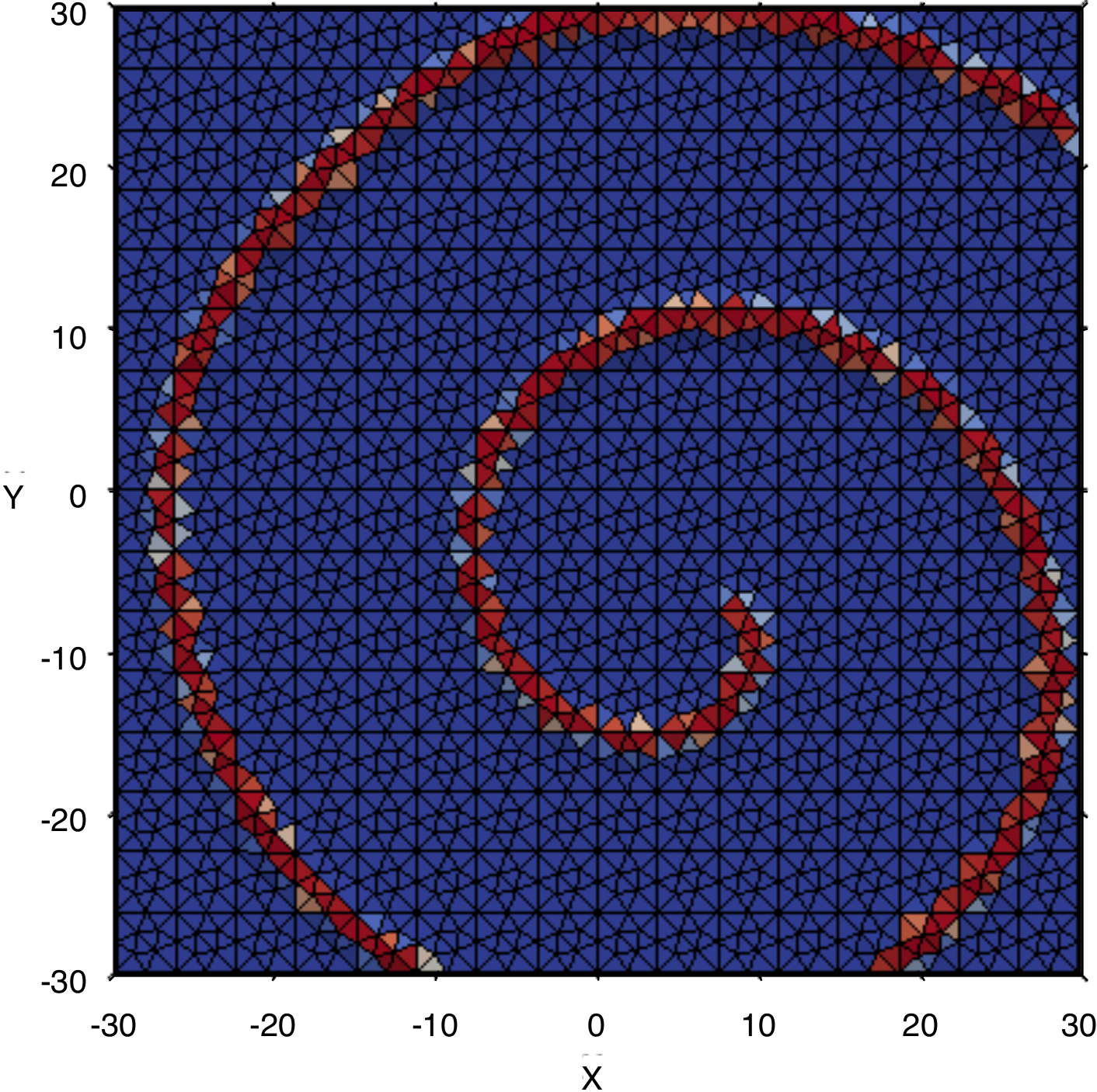} &\\
\texttt{Time:} $t=80$ & \texttt{Time:} $t=100$ &\\
\end{tabular}
\end{center}
\caption{Example 1, propagation of spiral waves of the Barkley model at different time levels with parameters detailed in text.}
\label{fig-t1}
\end{figure}

\par Figure \ref{fig-t1} shows the way that spiral waves were propagated in this instance, plotted at different times, with the simulation is done up to $t=100$. It can be seen that the development of the spiral wave involves a core at the medium's center, and this develops into a periodically rotating single spiral wave. The trajectory of this spiral wave runs symmetrically from origin to boundary, moving along the length to turn into stationary spirals. The observed behavior is called the reflection of the spiral at the boundary.


\subsection*{Example 2}\label{test2}
\par However, in certain instances, the spiral tip may collide with the boundary and then the spiral waves are annihilated instead of being reflected, especially if the initial displacement of the tip is close to the boundary as can be seen in the second example. In this case, we choose the parameters of the model as \cite{16}: $\rho=0.0208$, $a=0.52$, $b=0.05$ and $\mu=1$ on the square domain, which is of the length $L=7.5$. The initial conditions are given as
\[
u_{\rm ini}(x_1,x_2)=\begin{cases}
\hfill 0&, \quad \mbox{for } x_1<0 \mbox{ or }  x_2>5,\\
\bar u_0&, \quad \mbox{otherwise},
\end{cases}
\]
and
\[
v_{\rm ini}(x_1,x_2)=\begin{cases}
\hfill \bar v_0&, \quad \mbox{for } x_1<-1 \mbox{ and }  x_2<3,\\
0&, \quad \mbox{otherwise},
\end{cases}
\]
where $\bar u_0$ and $\bar v_0$ are as in Equation \eqref{eq2}, by putting $\alpha_1=3$, $\alpha_2=1$ and $\alpha_3=0.25$. Spiral solution for the Barkley model under the considered parameters is shown in Figure \ref{fig-t2} at different time levels, from $t=0.2$ to $t=3$. In this instance, with a small domain, the spiral commences taking shape at $t=1$, although the rotation is not yet complete. Following this, each part of the spiral will drift over to the nearest boundary and be absorbed. 

A dormant state will be achieved in both instances as when there are interactions between spiral waves and boundaries within excitable media, the trajectory of the wave is either reflected or destroyed by the interaction with the boundary, according to the parameters that have been chosen. We conclude that our numerical observations for the Barkley model obtained using HMM method are consistent with those observations in the previous studies \cite{16,18} in terms of spiral waves dynamics.

\begin{figure}[ht]
	\begin{center}
	\begin{tabular}{ccc}
\includegraphics[width=0.30\linewidth]{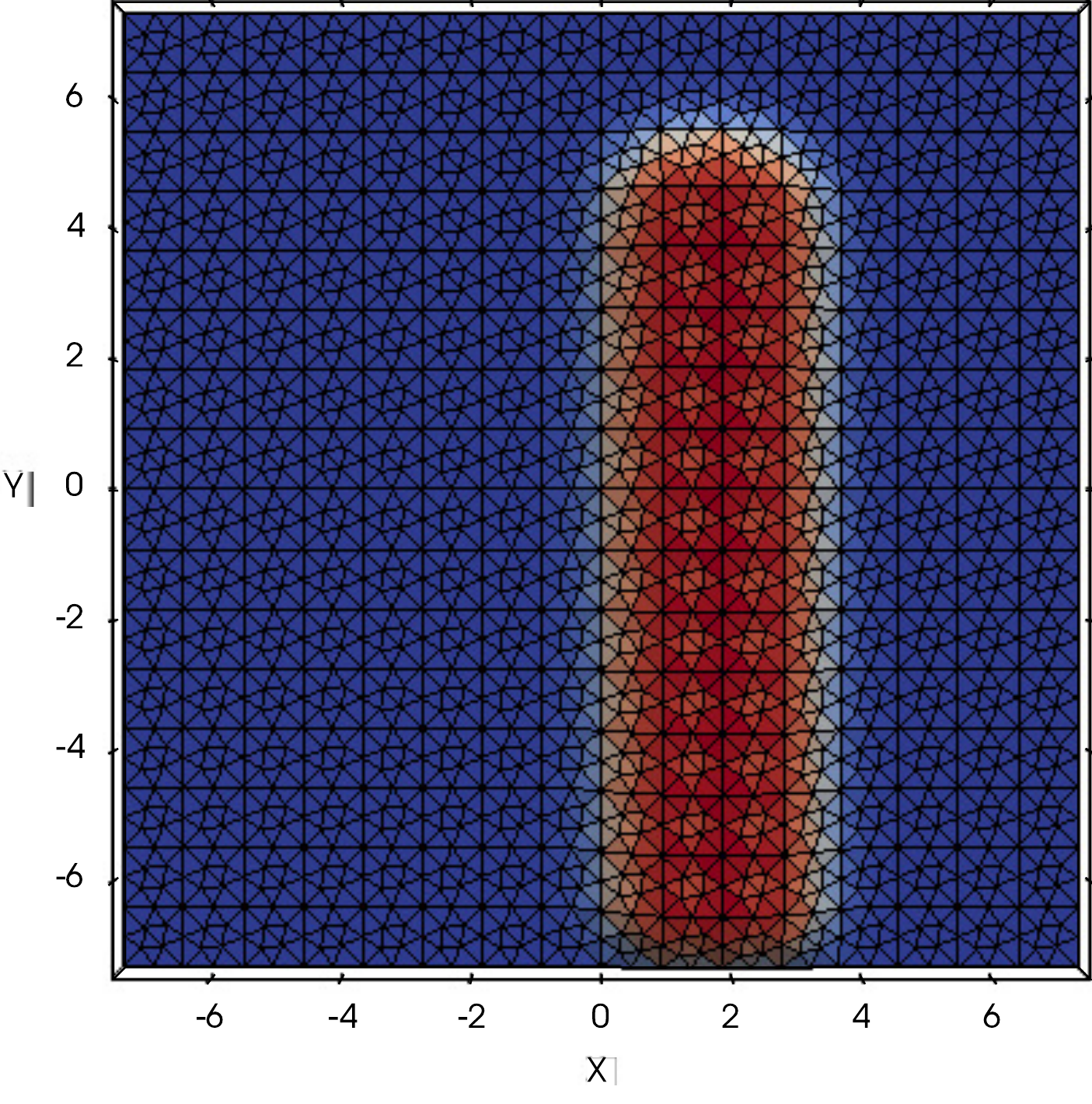} & \includegraphics[width=0.30\linewidth]{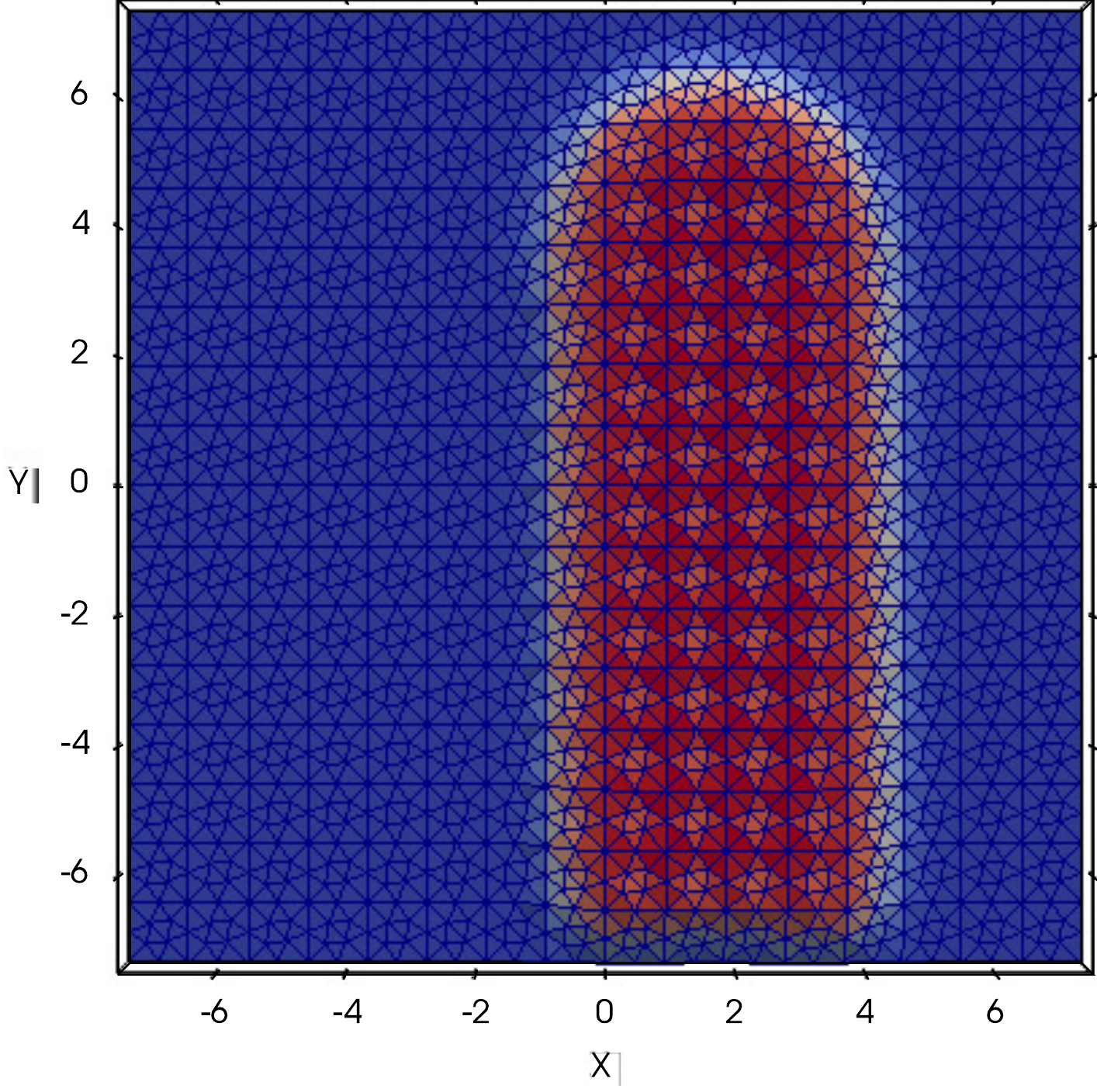} &
\includegraphics[width=0.30\linewidth]{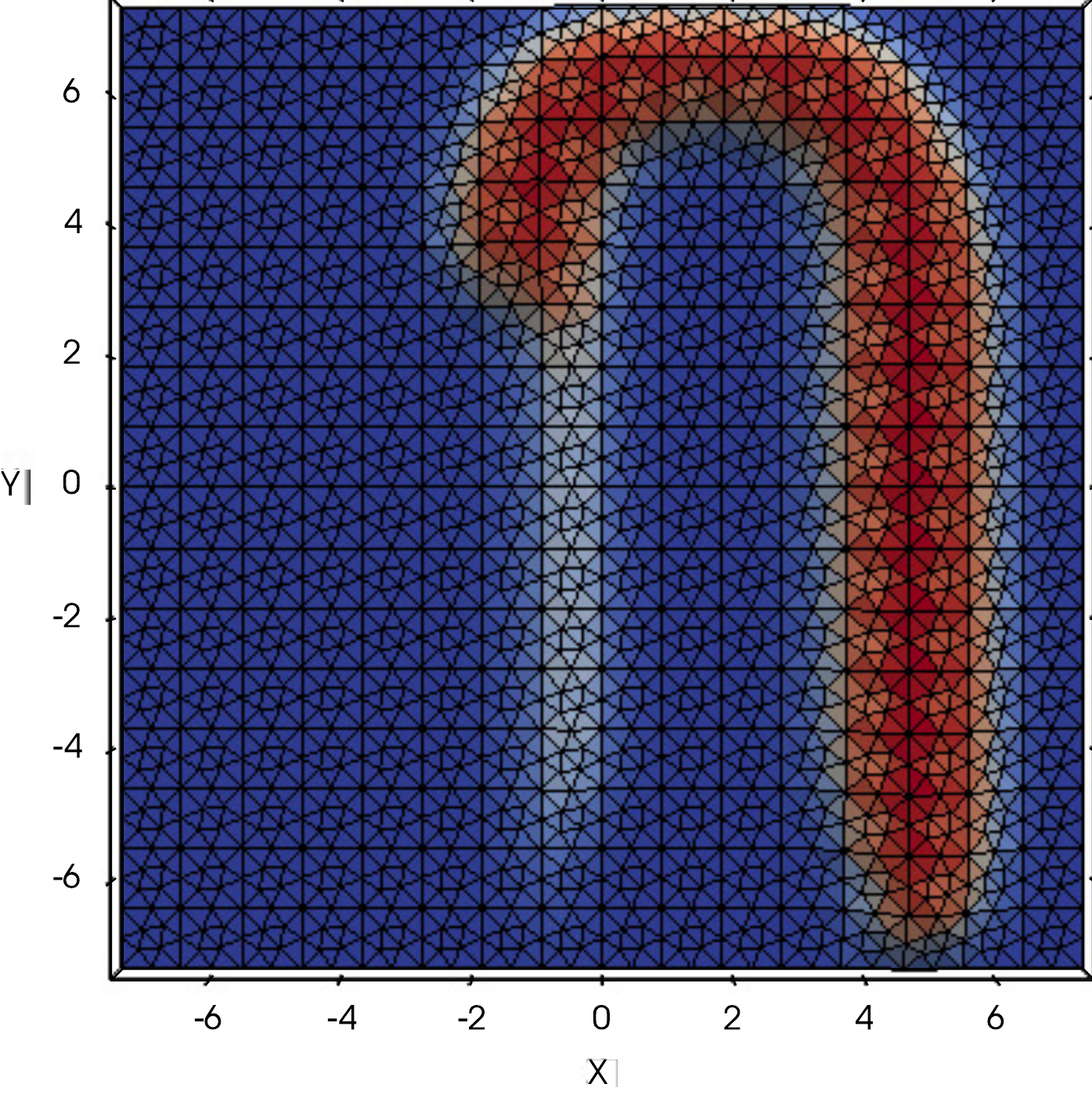}\\
\texttt{Time:} $t=5$ & \texttt{Time:} $t=10$ & \texttt{Time:} $t=20$\\
\includegraphics[width=0.30\linewidth]{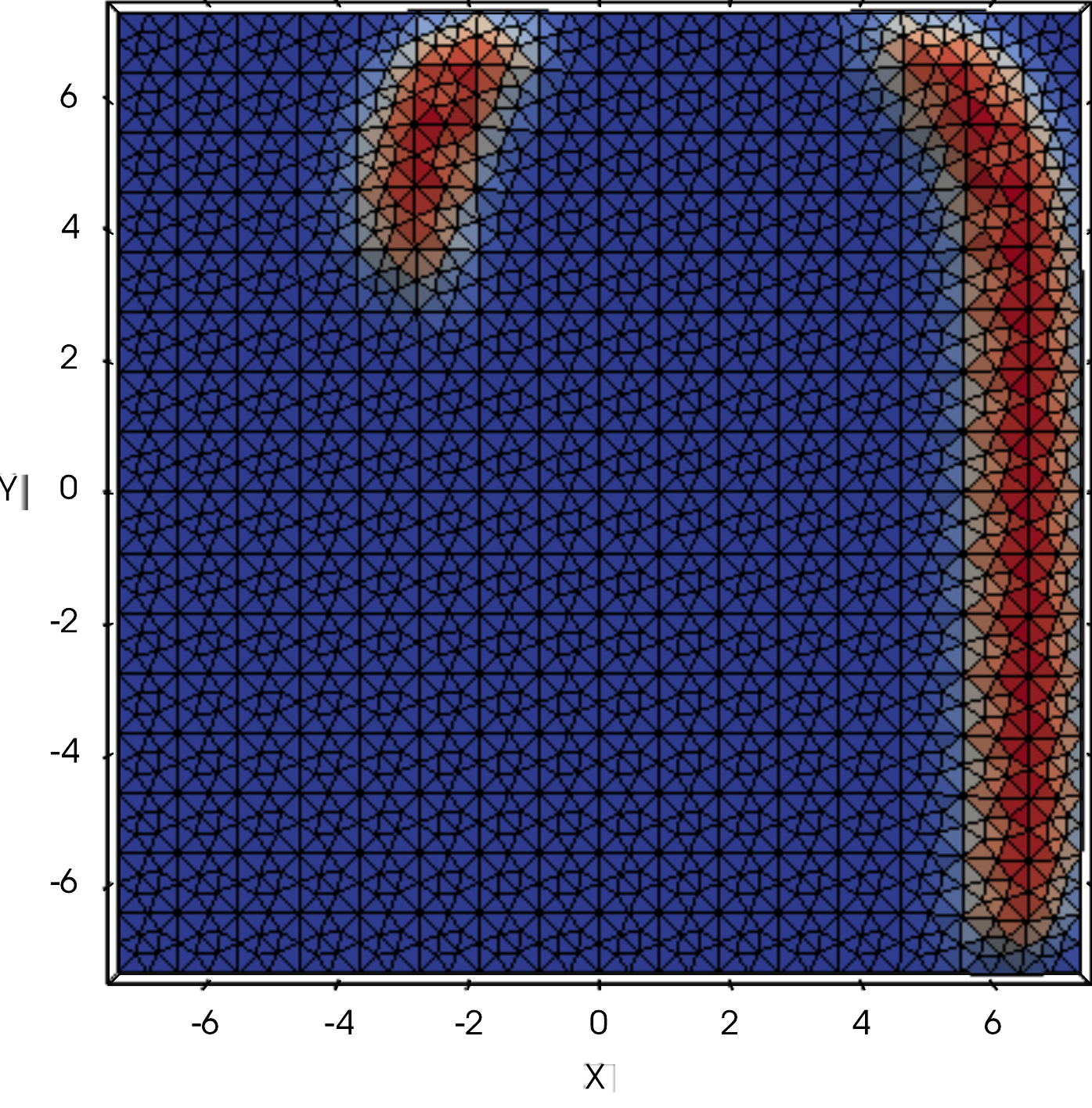} &
\includegraphics[width=0.30\linewidth]{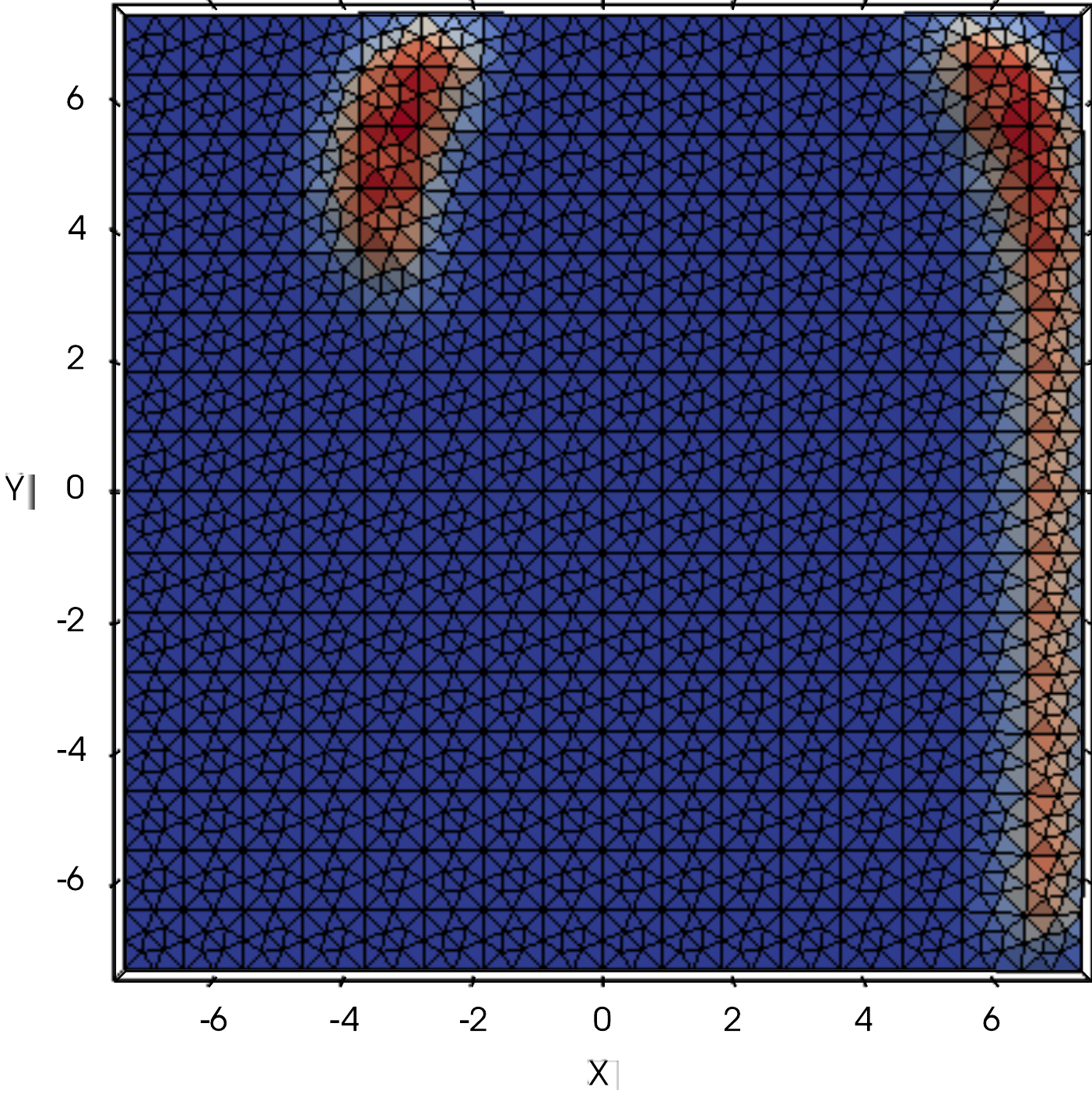} &
\includegraphics[width=0.30\linewidth]{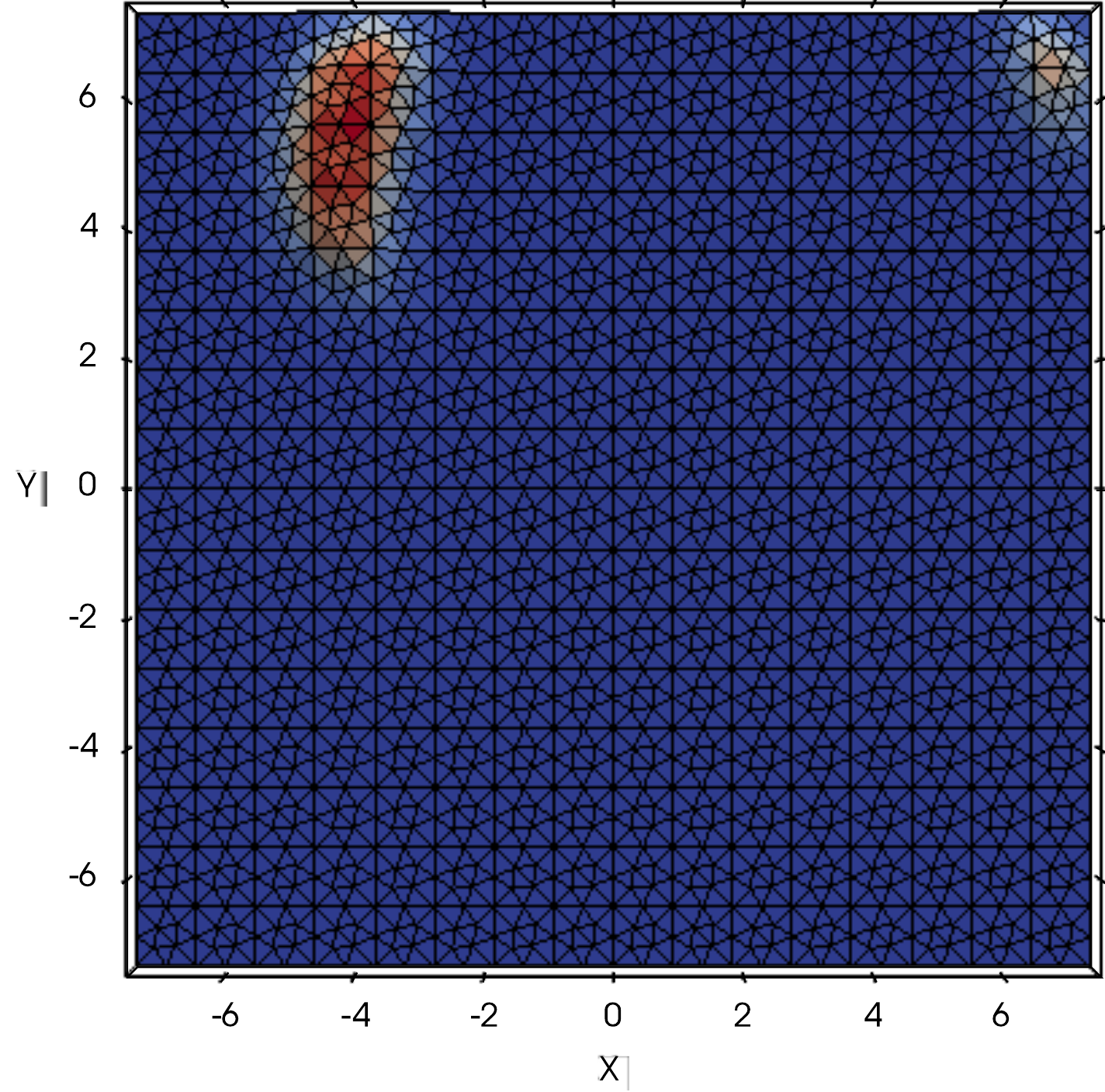}\\
\texttt{Time:} $t=30$ & \texttt{Time:} $t=40$ & \texttt{Time:} $t=60$\\
\includegraphics[width=0.30\linewidth]{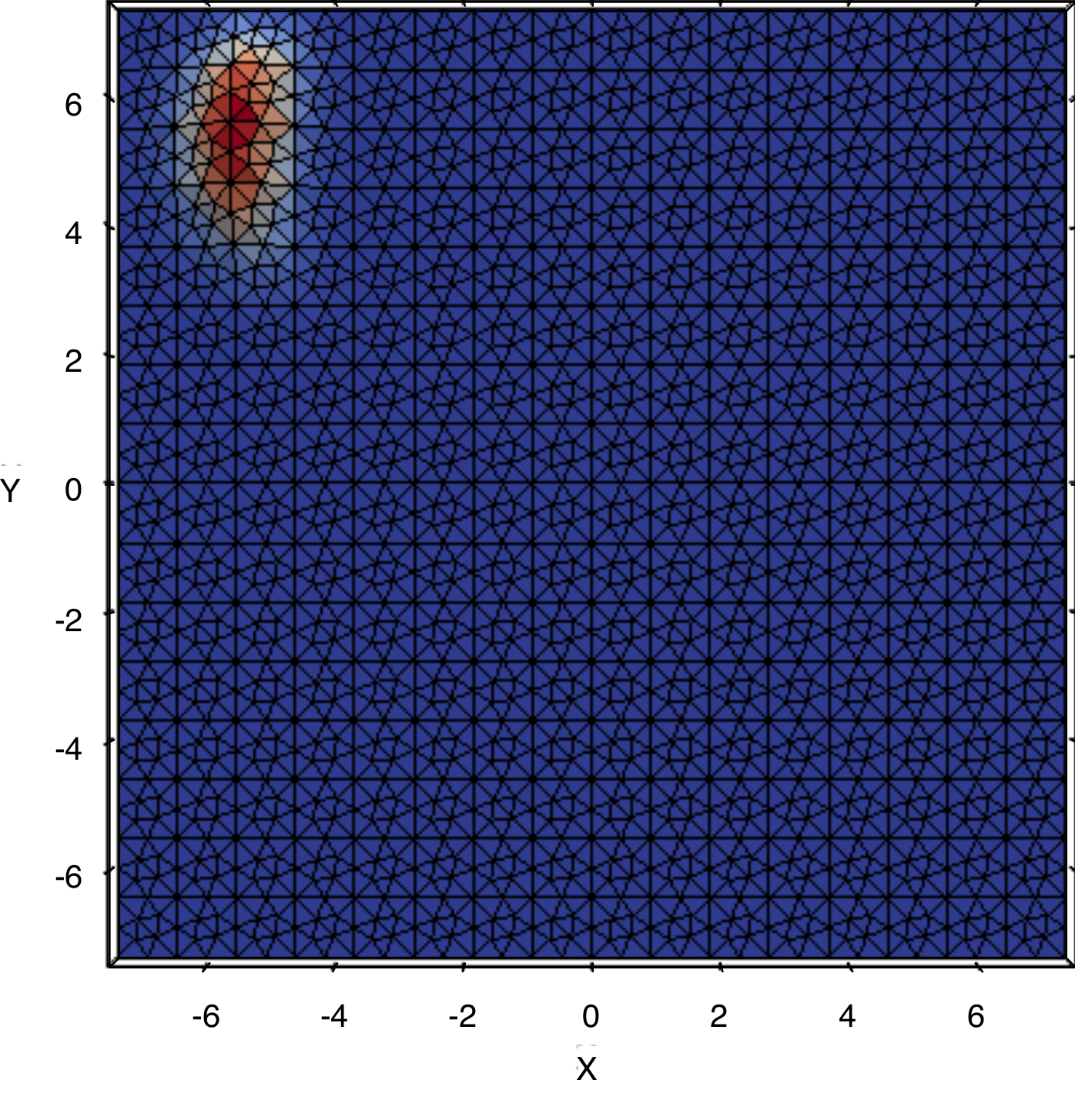} &
\includegraphics[width=0.30\linewidth]{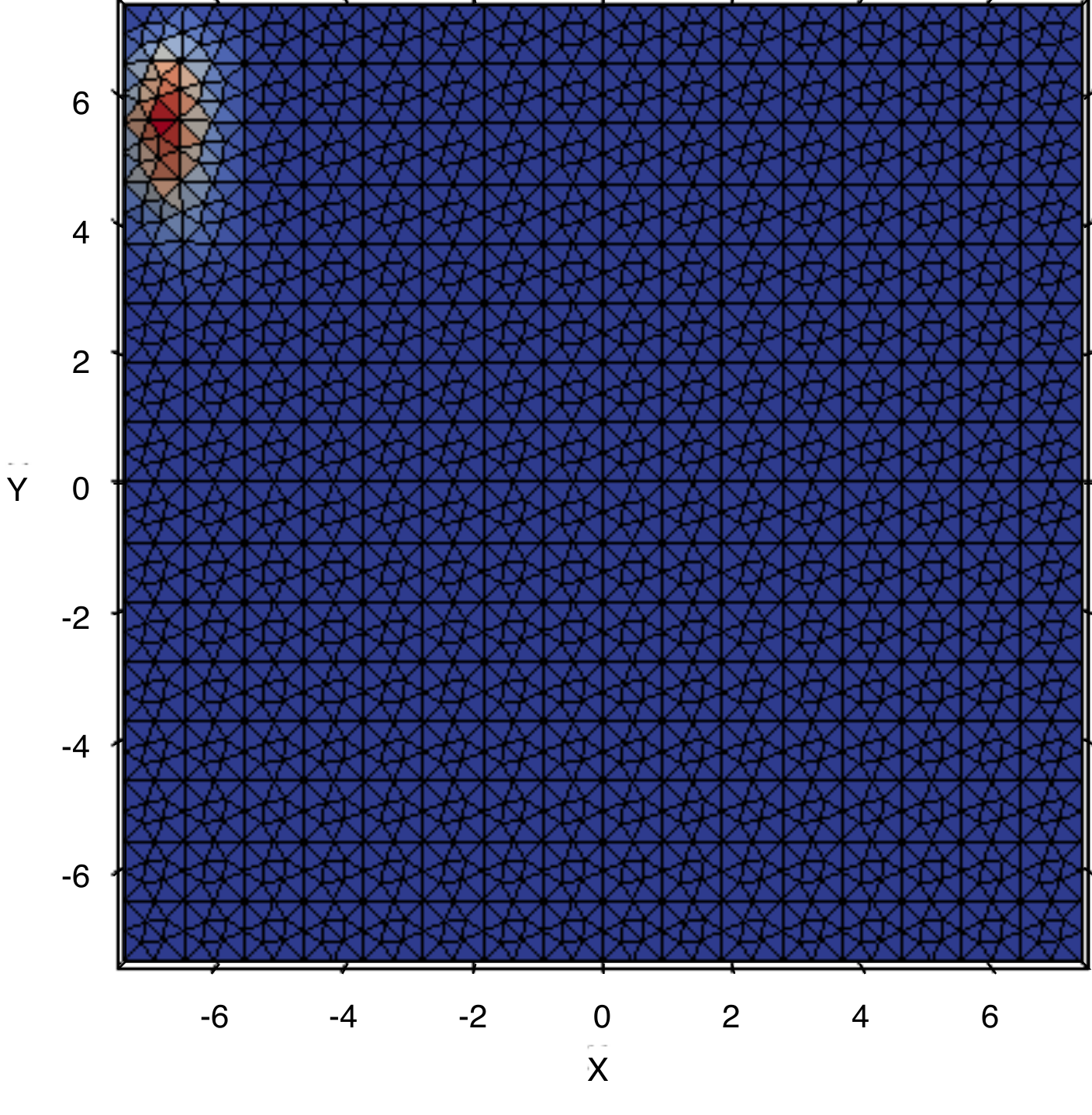} &\\
\texttt{Time:} $t=80$ & \texttt{Time:} $t=100$ &\\
\end{tabular}
\end{center}
\caption{Example 2, propagation of spiral waves of the Barkley model at different time levels with parameters detailed in text.} 
\label{fig-t2}
\end{figure}



\bibliographystyle{siam}
\bibliography{rm-ref}

\begin{thebibliography}{10}

\bibitem{40}
{\sc R.~R. Aliev and A.~V. Panfilov}, {\em A simple two-variable model of
  cardiac excitation}, Chaos Solitons Fractals, 7 (1996), pp.~293--301.

\bibitem{38}
{\sc Y.~Alnashri and J.~Droniou}, {\em Gradient schemes for the {S}ignorini and
  the obstacle problems, and application to hybrid mimetic mixed methods},
  Computers and Mathematics with Applications, 72 (2016), pp.~2788--2807.

\bibitem{37}
\leavevmode\vrule height 2pt depth -1.6pt width 23pt, {\em A gradient
  discretization method to analyze numerical schemes for nonlinear variational
  inequalities, application to the seepage problem}, SIAM Journal on Numerical
  Analysis, 56 (2018), pp.~2375--2405.

\bibitem{43}
{\sc H.~Alzubaidi and T.~Shardlow}, {\em Interaction of waves in a one
  dimensional stochastic pde model of excitable media}, Discrete And Continuous
  Dynamical Systems Series B, 18 (2013), pp.~1735--1754.

\bibitem{11}
{\sc F.~Amdjadi and J.~Gomatam}, {\em Spiral waves on static and moving
  spherical domains}, Journal of Computational and Applied Mathematics, 182
  (2005), pp.~472--486.

\bibitem{10}
{\sc D.~Barkley}, {\em A model for fast computer simulation of waves in
  excitable media}, Physica D, 49 (1991), pp.~61--70.

\bibitem{24}
{\sc M.~Bendahmane, R.~B\`{u}rger, and R.~Ruiz-Baier}, {\em A finite volume
  scheme for cardiac propagation in media with isotropic conductivities},
  Mathematics and Computers in Simulation, 80 (2010), pp.~1821--1840.

\bibitem{23}
{\sc M.~Bendahmane and K.~H. Karlsen}, {\em Convergence of a finite volume
  scheme for the bidomain model of cardiac tissue}, Applied Numerical
  Mathematics, 59 (2009), pp.~2266--2284.

\bibitem{18}
{\sc Y.~Bourgault, M.~Ethier, and V.~LeBlanc}, {\em Simulation of
  electrophysiological waves with an unstructured finite element method},
  Mathematical Modelling and Numerical Analysis, 37 (2003), pp.~649--661.

\bibitem{9}
{\sc N.~Britton}, {\em Reaction-Diffusion Equations and Their Applications to
  Biology}, Academic Press, New York, London, 1986.

\bibitem{25}
{\sc R.~B\"{u}rger, R.~Ruiz-Baier, and K.~Schneider}, {\em Adaptive
  multiresolution methods for the simulation of waves in excitable media},
  Journal of Scientific Computing, 43 (2010), pp.~261--290.

\bibitem{4}
{\sc B.~H. C.~Diks and J.~Degoede}, {\em Spiral wave dynamics}, Chaos Solitom
  and Fractals, 5 (1995), pp.~645--660.

\bibitem{2}
{\sc D.~J. Christini and L.~Glass}, {\em Introduction: Mapping and control of
  complex cardiac arrhythmias}, Chaos, 12 (2002), pp.~732--739.

\bibitem{22}
{\sc Y.~Coudi\`{e}re and C.~Pierre}, {\em Stability and convergence of a finite
  volume method for two systems of reaction-diffusion equations in
  electro-cardiology}, Nonlinear Analysis: Real World Applications, 7 (2006),
  pp.~916 -- 935.

\bibitem{26}
{\sc Y.~Coudi\`{e}re and R.~Turpault}, {\em Very high order finite volume
  methods for cardiac electrophysiology}, Computers and Mathematics with
  Applications, 74 (2017), pp.~684--700.

\bibitem{13}
{\sc Devanand and J.~C. Kalita}, {\em Hoc simulation of barkley model in
  excitable media}, AIP Conference Proceedings, 1975 (2018), p.~030011.

\bibitem{44}
{\sc J.~Droniou}, {\em Introduction to discrete functional analysis techniques
  for the numerical study of diffusion equations with irregular data}, in
  Proceedings of the 17th Biennial Computational Techniques and Applications
  Conference, CTAC-2014, J.~Sharples and J.~Bunder, eds., vol.~56 of ANZIAM J.,
  2015, pp.~C101--C127.

\bibitem{36}
{\sc J.~Droniou and R.~Eymard}, {\em Uniform-in-time convergence of numerical
  methods for non-linear degenerate parabolic equations}, Numer. Math., 132
  (2016), pp.~721--766.

\bibitem{30}
{\sc J.~Droniou, R.~Eymard, T.~Gallou\"et, C.~Guichard, and R.~Herbin}, {\em
  The gradient discretisation method}, Mathematics \& Applications, Springer,
  Heidelberg, 2018.

\bibitem{42}
{\sc J.~Droniou, R.~Eymard, T.~Gallou\"et, and R.~Herbin}, {\em A unified
  approach to mimetic finite difference, hybrid finite volume and mixed finite
  volume methodss}, Mathematical Models and Methods in Applied Sciences, 20
  (2010), pp.~265--295.

\bibitem{31}
{\sc J.~Droniou, R.~Eymard, and R.~Herbin}, {\em Gradient schemes: a generic
  framework for the discretisation of linear, nonlinear and nonlocal elliptic
  and parabolic problems}, Mathematical Models and Methods in Applied Sciences,
  23 (2013), pp.~2395--2432.

\bibitem{33}
{\sc J.~Droniou, R.~Eymard, and R.~Herbin}, {\em Gradient schemes: generic
  tools for the numerical analysis of diffusion equations}, {M2AN} Math. Model.
  Numer. Anal., 50 (2016), pp.~749--781.

\bibitem{35}
{\sc R.~Eymard, P.~Feron, T.~Gallou\"et, R.~Herbin, and C.~Guichard}, {\em
  Gradient schemes for the {S}tefan problem}, International Journal On Finite
  Volumes, 10s (2013).

\bibitem{32}
{\sc R.~Eymard, C.~Guichard, and R.~Herbin}, {\em Small-stencil {3D} schemes
  for diffusive flows in porous media}, ESAIM. Mathematical Modelling and
  Numerical Analysis, 46 (2012), pp.~265--290.

\bibitem{34}
{\sc R.~Eymard, C.~Guichard, R.~Herbin, and R.~Masson}, {\em Gradient schemes
  for two-phase flow in heterogeneous porous media and {R}ichards equation},
  ZAMM Z. Angew. Math. Mech., 94 (2014), pp.~560--585.

\bibitem{7}
{\sc F.~H. Fentona, E.~M. Cherry, H.~M. Hastings, and S.~J. Evans}, {\em
  Multiple mechanisms of spiral wave breakup in a model of cardiac electrical
  activity}, Chaos, 12 (2002), pp.~852--892.

\bibitem{39}
{\sc R.~FitzHugh}, {\em Impulses and physiological states in theoretical models
  of nerve membrane}, Biophysical Journal, 1 (1961), pp.~445--466.

\bibitem{28}
{\sc M.~Gomez-Gesteira, A.~M. nuzuri, V.~P.-M. nuzuri, and V.~Perez-Villar},
  {\em Boundary imposed spiral drift}, Physical Review E, 53 (1996),
  pp.~5480--5483.

\bibitem{5}
{\sc A.~Goryachev and R.~Kapral}, {\em Spiral waves in chaotic systems},
  Physical Review Letters, 76 (1996), pp.~1619--1622.

\bibitem{20}
{\sc D.~Harrild and C.~Henriquez}, {\em A finite volume model of cardiac
  propagation}, Annals of Biomedical Engineering, 25 (1997), pp.~315--334.

\bibitem{15}
{\sc A.~Karma}, {\em Meandering transition in two-dimensional excitable media},
  Physical Review Letters, 65 (1990), pp.~2824--2828.

\bibitem{6}
{\sc J.~Keener}, {\em A geometrical theory for spiral waves in excitable
  media}, SIAM Journal on Applied Mathematics, 46 (1986), pp.~1039--1059.

\bibitem{1}
\leavevmode\vrule height 2pt depth -1.6pt width 23pt, {\em Arrhythmias by
  dimension}, Proceedings of Symposia in Applied Mathematics, 59 (2002),
  pp.~57--81.

\bibitem{8}
{\sc J.~P. Keener and J.~Sneyd}, {\em Mathematical Physiology},
  Springer-Verlag, 1998.

\bibitem{12}
{\sc J.~Li and J.~Li}, {\em High-order compact difference methods for
  simulating wave propogation in excitable media}, International Journal Of
  Numerical Analysis And Modeling, Series B, 5 (2014), pp.~339--346.

\bibitem{27}
{\sc D.~Olmos and B.~Shizgal}, {\em Annihilation and reflection of spiral waves
  at a boundary for the beeler-reuter model}, Physical Review E, 77 (2008),
  pp.~33--50.

\bibitem{16}
\leavevmode\vrule height 2pt depth -1.6pt width 23pt, {\em Pseudospectral
  method of solution of the fitzhugh nagumo equation}, Mathematics and
  Computers in Simulation, 79 (2009), pp.~2258--2278.

\bibitem{14}
{\sc J.~Ramos}, {\em Spiral wave break-up and planar front formation in
  two-dimensional reactive diffusive media with straining}, Applied Mathematics
  and Computation, 154 (2004), pp.~697--711.

\bibitem{17}
{\sc J.~{Rogers}, M.~{Courtemanche}, and A.~{McCulloch}}, {\em {Finite element
  methods for modelling impulse propagation in the heart.}}, in {Computational
  biology of the heart. Workshop on Whole Heart Modeling, 11th--13th February
  1994, Utrecht, the Netherlands}, Chichester: Wiley, 1997, pp.~217--233.

\bibitem{19}
{\sc J.~M. Rogers and A.~D. McCulloch}, {\em A collocation-galerkin finite
  element model of cardiac action potential propagation}, IEEE Transactions on
  Biomedical Engineering, 41 (1994), pp.~743--757.

\bibitem{45}
{\sc J.~Smoller}, {\em Shock Waves and Reaction--Diffusion Equations},
  Springer, 1983.

\bibitem{21}
{\sc M.~Trew, I.~L. Grice, B.~Smaill, and A.~Pullan}, {\em A finite volume
  method for modeling discontinuous electrical activation in cardiac tissue},
  Annals of Biomedical Engineering, 33 (2005), pp.~590--602.

\bibitem{3}
{\sc A.~Winfree}, {\em Varieties of spiral wave behavior: an experimentalist's
  approach to the theory of excitable media}, Chaos, 1 (1991), pp.~303--334.

\bibitem{29}
{\sc Y.~A. Yermakova and A.~M. Pertsov}, {\em Interaction of rotating spiral
  waves with a boundary}, Biophys, 31 (1986), pp.~932--940.

\bibitem{41}
{\sc A.~M. Zhabotinskya}, {\em A history of chemical oscillations and waves},
  CHAOS, 1 (1991), pp.~379--386.

\end{thebibliography}

\end{document}